\documentclass[11pt,leqno]{amsart}
 \usepackage{graphicx}    
 \usepackage{amsmath}
\usepackage{amssymb}
\usepackage[T1]{fontenc}
\usepackage[latin1]{inputenc}
\usepackage[english]{babel}
\usepackage[arrow, matrix, curve]{xy}
\usepackage{amsthm}
\usepackage{amsmath,amscd}
\usepackage{mathrsfs}
\usepackage{epic,eepic}
\numberwithin{equation}{section}

\DeclareMathOperator\GL{GL}
\DeclareMathOperator\id{id}
\DeclareMathOperator\Hom{Hom}

\DeclareMathOperator\rk{rk}
\DeclareMathOperator\dete{det}

\DeclareMathOperator\gr{gr}
\DeclareMathOperator\Gr{Gr}
\DeclareMathOperator\im{im}

\DeclareMathOperator\coker{coker}
\DeclareMathOperator\Spec{Spec}
\DeclareMathOperator\Spf{Spf}

\DeclareMathOperator\Spa{Spa}

\DeclareMathOperator\pr{pr}
\DeclareMathOperator\Fil{Fil}
\DeclareMathOperator\Isoc{Isoc}
\DeclareMathOperator\Res{Res}
\DeclareMathOperator\rig{rig}
\DeclareMathOperator\ad{ad}
\DeclareMathOperator\Rig{Rig}
\DeclareMathOperator\Ad{Ad}
\DeclareMathOperator\lft{lft}
\DeclareMathOperator\Frac{Frac}

\renewcommand{\phi}{\varphi}

\newcommand{\Acal}{\mathscr{A}}
\newcommand{\Bcal}{\mathscr{B}}
\newcommand{\Ccal}{\mathcal{C}}

\newcommand{\Ecal}{\mathcal{E}}
\newcommand{\Fcal}{\mathcal{F}}
\newcommand{\Gcal}{\mathcal{G}}

\newcommand{\Mcal}{\mathcal{M}}
\newcommand{\Ncal}{\mathcal{N}}
\newcommand{\Ocal}{\mathcal{O}}

\newcommand{\Xcal}{\mathcal{X}}

\newcommand{\sheafHom}{\mathscr{H}om}

\newcommand{\Q}{\mathbb{Q}}
\newcommand{\R}{\mathbb{R}}
\newcommand{\Z}{\mathbb{Z}}

\newcommand{\Abb}{\mathbb{A}}
\newcommand{\boldB}{\mathbb{B}}
\newcommand{\Fbb}{\mathbb{F}}
\newcommand{\Gbb}{\mathbb{G}}
\newcommand{\Pbb}{\mathbb{P}}
\newcommand{\Ubb}{\mathbb{U}}

\newcommand{\Cfrak}{\mathfrak{C}}
\newcommand{\Dfrak}{\mathfrak{D}}
\newcommand{\Mfrak}{\mathfrak{M}}

\newcommand{\mfrak}{\mathfrak{m}}

\newtheorem{theo}{Theorem}[section]
\newtheorem{lem}[theo]{Lemma}
\newtheorem{prop}[theo]{Proposition}
\newtheorem{cor}[theo]{Corollary}
\theoremstyle{remark}
\newtheorem{rem}[theo]{Remark}
\theoremstyle{remark}
\newtheorem{expl}[theo]{Example}
\theoremstyle{definition}
\newtheorem{defn}[theo]{Definition}

\begin{document}
\title[Families of filtered $\phi$-modules and crystalline representations]{On arithmetic families of filtered $\phi$-modules \\ and crystalline representations}
\author[E. Hellmann]{Eugen Hellmann}
\begin{abstract}
We consider stacks of filtered $\phi$-modules over rigid analytic spaces and adic spaces.
We show that these modules parametrize $p$-adic Galois representations of the absolute Galois group of a $p$-adic field with varying coefficients over an open substack
containing all classical points. 
Further we study a period morphism (defined by Pappas and Rapoport) from a stack parametrising integral data and determine the image of this morphism.
\end{abstract}
\maketitle

\section{Introduction}
In $p$-adic Hodge theory one considers filtered $\phi$-modules as a category of linear algebra data describing \emph{crystalline} Galois representations. 
More precisely, there is an equivalence of categories between crystalline representations and weakly admissible filtered $\phi$-modules (cf. \cite[Theorem A]{ColmezFont}),
where weak admissibility is a certain condition relating the slopes of the Frobenius $\Phi$ with the filtration.

In their book \cite{RapoZink}, Rapoport and Zink consider geometric families of these modules. More precisely they fix an isocrystal over $W(\bar\Fbb_p)[1/p]$ and consider for a given coweight $\nu$ the flag variety $\Fcal_\nu$ over $W(\bar\Fbb_p)[1/p]$, where $\nu$ prescribes the jumps of the filtration. They show that there is an admissible open subspace $\Fcal_\nu^{\rm wa}\subset \Fcal_\nu^{\rig}$ parametrizing those filtered $\phi$-modules that are weakly admissible with respect to the given isocrystal.

In this paper we consider \emph{arithmetic families} of filtered $\phi$-modules. That is, we fix a local field $K$ of characteristic $0$, and study a stack on the category of rigid analytic spaces parametrising $K$-filtered $\phi$-modules with coefficients defined by Pappas and Rapoport in \cite{phimod}. The difference with the geometric situation is that we do not fix an isocrystal (i.e. the Frobenius may vary) and the filtration also varies. We define a notion of weak admissibility in this context, and show that the weakly admissible locus in this stack is an open substack. Here we work in the category of adic spaces introduced by Huber (see \cite{Huber} for example), instead of the category of rigid analytic spaces. 

The work of Kedlaya and Liu \cite{KedlayaLiu} shows that the category of Berkovich spaces (see \cite{Berko} for example) is not the right category in which to address  these questions, since the locus where certain $\phi$-modules over the Robba ring are \'etale is not an open subspace. This phenomenon appears here again, since the weakly admissible locus will not be a Berkovich space, as we show by an explicit example (Example $\ref{noBerko}$).
Kedlaya and Liu formulate the local \'etaleness of $\phi$-modules over the Robba ring in terms of rigid analytic spaces. However, we cannot apply their result, since a covering by admissible open subsets is not necessarily an admissible covering. This is the reason why we are forced to work in the category of adic spaces. In fact, we will generalize the theorem of Kedlaya and Liu to the setting of adic spaces.

In the geometric setting of \cite{RapoZink}, there is a period morphism from the moduli space of $p$-divisible groups to a certain period domain $\Fcal_\nu^{\rm wa}$. The image of this morphism was described by Hartl in \cite{Hartl1}, \cite{Hartl2} and Faltings, \cite{Faltings}. 
As an analogue to the moduli space of $p$-divisible groups, Pappas and Rapoport define in \cite{phimod} a formal stack parametrising modules that appear in integral $p$-adic Hodge theory developed by Breuil and Kisin (see \cite{Breuil} for example). They also define an analogue of the period map of \cite{RapoZink}. In this paper we will determine the image of this period map, confirming a conjecture of Pappas and Rapoport.

Finally we construct an open substack of the stack of weakly admissible filtered $\phi$-modules and a family of crystalline representations on this stack, i.e. a vector bundle with a continuous Galois action which is crystalline. We show that this subspace is in fact universal for families of crystalline representations.

Our main results are as follows. Let $K$ be a finite extension of $\Q_p$ and denote by $K_0$ the maximal unramified extension of $\Q_p$ inside $K$. We fix an integer $d>0$ and a dominant coweight $\nu$ of the algebraic group $\Res_{K/\Q_p}\GL_d$ with associated reflex field $E$. We denote by $\Dfrak_\nu$ the fpqc-stack on the category of rigid analytic spaces over $E$ (or, slightly more generally, on the category of adic spaces locally of finite type) whose $X$-valued points are triples $(D,\Phi,\Fcal^\bullet)$ with a locally free $\Ocal_X\otimes_{\Q_p}K_0$-module $D$, a semi-linear automorphism $\Phi$ and a filtration of $D\otimes_{\Q_p}K$ which is of type $\nu$. It is easy to see that this stack is an \emph{Artin stack}, where we define an Artin stack on the category of rigid spaces (or adic spaces) exactly as on the category of schemes.
We will define a notion of \emph{weak admissibility} for all points of the adic space generalizing the usual notion (cf. \cite[3.4]{ColmezFont}) at the rigid analytic points.
\begin{theo}\label{maintheo1}
The weakly admissible locus is an open substack $\Dfrak_\nu^{\rm wa}\subset \Dfrak_\nu$ on the category of adic spaces locally of finite type over $E$.
\end{theo} 
We will also consider the following stack $\widehat{\Ccal}_K$ on the category of $\Z_p$-schemes on which $p$ is nilpotent. The $S$-valued points of $\widehat{\Ccal}_K$ are tuples $(\Mfrak,\Phi)$, where $\Mfrak$ is a (fpqc-locally on $S$) free $\Ocal_S\widehat{\otimes}_{\Z_p}W[[u]]$-module together with a semi-linear injection $\Phi:\Mfrak\rightarrow\Mfrak$ whose cokernel is killed by the minimal polynomial over $K_0$ of some fixed uniformizer of $K$. Here $W$ denotes the ring of integers of $K_0$.
The category of those modules was introduced by Breuil and studied by Kisin in order to describe finite flat group schemes of $p$-power order over $\Spec\,\Ocal_K$ and hence to describe $p$-divisible groups in the limit. Given a miniscule cocharacter $\nu$, Pappas and Rapoport define a closed substack $\widehat{\Ccal}_{K,\nu}$ of $\widehat{\Ccal}_K$ by posing an extra condition on the cokernel of the semi-linear injection $\Phi$. For a formal scheme $\Xcal$ over $\Z_p$ they also define a \emph{period map}
\[\Pi(\Xcal):\widehat{\Ccal}_{K,\nu}(\Xcal)\longrightarrow \Dfrak_\nu(\Xcal^{\rig}),\]
that maps the $\Z_p$-point associated with a $p$-divisible group over $\Ocal_K$ to the filtered isocrystal of the $p$-divisible group. We will show that the map indeed factors through the weakly admissible locus and we describe the image of the period map as follows. In \cite{crysrep}, Kisin shows that there is an equivalence between the category of filtered $\phi$-modules and a certain category of vector bundles on the open unit disc together with a $\phi$-linear map. We will show that the equivalence of categories generalizes to families. As the notion of weak admissibility (for filtered isocrystals over $\Q_p$) translates to the property of being \'etale over the Robba ring, we will define a substack $\Dfrak_\nu^{\rm int}$ of $\Dfrak_\nu$ consisting of those filtered $\phi$-modules whose associated vector bundle is \'etale.
\begin{theo}\label{maintheo3}
Let $\nu$ be a miniscule cocharacter. The stack $\Dfrak_\nu^{\rm int}$ is an open substack of $\Dfrak_\nu^{\rm wa}$ {\rm (}on the category of adic spaces locally of finite type over the reflex field of $\nu${\rm )}. 
It is the image of the period morphism in the sense that a morphism $X\rightarrow \Dfrak_\nu^{\rm wa}$, defining $(D,\Phi,\Fcal^\bullet)\in \Dfrak_\nu(X)$, 
factors through $\Dfrak_\nu^{\rm int}$ if and only if there exists a covering $(U_i)_{i\in I}$ of $X$ and $p$-adic formal schemes $\mathcal{U}_i$ together with $(\Mfrak_i,\Phi_i)\in\widehat{\Ccal}_{K,\nu}(\mathcal{U}_i)$ such that $\mathcal{U}_i^{\rm ad}=U_i$ and 
\[\Pi(\mathcal{U}_i)(\Mfrak_i,\Phi_i)=(D,\Phi,\Fcal^\bullet)|_{U_i}.\]
\end{theo}
Finally we go back to crystalline representations. We consider vector bundles $\Ecal$ on an adic space $X$ endowed with an action of the absolute Galois group $G_K$ of $K$. As the structure sheaf $\Ocal_X$ is a sheaf of topological rings and $\Ecal$ is a vector bundle, all spaces of sections $\Gamma(U,\Ecal)$ of $\Ecal$ over some open subset $U$ have a natural topology and hence the notion of a continuous $G_K$-action on $\Ecal$ makes sense. A vector bundle $\Ecal$ with a continuous $G_K$-action such that 
\[D_{\rm cris}(\Ecal)=(\Ecal\widehat{\otimes}_{\Q_p}B_{\rm cris})^{G_K}\]
is locally on $X$ free of rank $d=\rk\Ecal$ over $\Ocal_X\otimes_{\Q_p}K_0$ is called a \emph{family of crystalline representations on} $X$.
\begin{theo}
Let $\nu$ be a miniscule cocharacter of $\Res_{K/\Q_p}\GL_d$. Then the groupoid of families of crystalline representations of Hodge-Tate weights $\nu$ is an open substack $\Dfrak_\nu^{\rm adm}$ of $\Dfrak_\nu^{\rm int}$.
\end{theo}
The following diagram summarizes the stacks that appear in this paper in the case of a miniscule cocharacter $\nu$.
\[\begin{xy}
\xymatrix{
 & \Dfrak_\nu^{\rm adm} \ar@^{(->}[d]\\
\widehat{\Ccal}_{K,\nu}^{\rm ad}\ar[dd]\ar[ddr]^\Pi \ar[r]& \Dfrak_\nu^{\rm int}\ar@^{(->}[d]\\
& \Dfrak_\nu^{\rm wa}\ar@^{(->}[d]\\
\Cfrak_\nu\ar[r]^\cong & \Dfrak_\nu.
}\end{xy}\]
Here $\Cfrak_\nu$ is a stack of vector bundles on the open unit disc that appears as an intermediate step of the period morphism. By the adification of the stack $\widehat{\Ccal}_{K,\nu}$ we mean the stack mapping an adic space $X$ to the limit of $\widehat{\Ccal}_{K,\nu}(\mathcal{X})$ for all formal models $\Xcal$ of $X$. 
The map $\widehat{\Ccal}_{K,\nu}^{\rm ad}\rightarrow \Dfrak_\nu^{\rm int}$ then is the localization map of $\widehat{\Ccal}_{K,\nu}^{\rm ad}$ to its isogeny category.
The vertical arrows on the right are all open embeddings and $\Dfrak_\nu^{\rm adm}(F)=\Dfrak_\nu^{\rm int}(F)=\Dfrak_\nu^{\rm wa}(F)$ for all finite extensions $F$ of the reflex field of $\nu$.
In general the inclusion of $\Dfrak_\nu^{\rm adm}$ in $\Dfrak_\nu^{\rm int}$ is a strict inclusion, as shown in Example $\ref{unramchars}$. It is an open question whether $\Dfrak_\nu^{\rm int}=\Dfrak_\nu^{\rm wa}$, compare Remark $\ref{intgleichwa}$

\noindent We now outline the structure of this paper.\\
In section 2 we will recall some basic facts and concepts from $p$-adic Hodge-theory and the theory of adic spaces.\\
In section 3 we introduce filtered isocrystals with coefficients in a valuated field. These objects will appear as the fibers of our families. We introduce the notions of semi-stability and weak admissibility, and establish a Harder-Narasimhan formalism.\\
In section 4 we define the stacks of filtered isocrystals and prove our first main result, Theorem $\ref{maintheo1}$.\\
In section 5 we study the relation of filtered isocrystals and vector bundles on the open unit disc which was introduced by Kisin.\\ 
In section 6 we discuss the notion of being \'etale over the Robba ring in families, slightly generalizing a result of Kedlaya and Liu about local \'etaleness of those modules. This allows us to determine the image of the period map in section 7 and prove Theorem $\ref{maintheo3}$.\\
In section 8 we construct a family of crystalline Galois representations on an open substack and show that this family is universal.

{\bf Acknowledgements}: I thank my advisor M. Rapoport for his interest and advice. Further I would like to thank R. Huber for answering my questions about adic spaces and G. Faltings for pointing out some small errors in an earlier version of this paper.
I also acknowledge the hospitality of the Institute Henri Poincare during the Galois trimester in spring 2010 and the hospitality of Harvard University in fall 2010,
where part of this work was done.
This work was supported by the SFB/TR 45 "Periods, Moduli Spaces and Arithmetic of Algebraic Varieties" of the DFG (German Research Foundation).

\section{Preliminaries}
Throughout the whole paper we fix the following notations:
Let $K$ be a finite extension of $\Q_p$ and write $\Ocal_K$ for its ring of integers.
Fix a uniformizer $\pi \in\Ocal_K$ and write $k=\Ocal_K/\pi\Ocal_K$ for the residue field.
Let $W=W(k)$ be the Witt ring of $k$ and $K_0=W[1/p]$ the maximal unramified extension of $\Q_p$ inside $K$.
Further we denote by $E(u)\in W[u]$ the minimal polynomial of $\pi$ over $K_0$.

Fix an algebraic closure $\bar K$ of $K$ and write $G_K={\rm Gal}(\bar K/K)$ for the absolute Galois group.
Further we choose a compatible system $\pi_n$ of $p^n$-th roots of $\pi$ in $\bar K$ and write $K_\infty$ for the field obtained from $K$ by adjoining the $\pi_n$.
Let $G_{K_\infty}={\rm Gal}(\bar K/K_\infty)$ denote the absolute Galois group of $K_\infty$.

\subsection{Some $p$-adic Hodge-theory}

For further use we define some rings of $p$-adic Hodge theory used in Kisin's papers \cite{Kisin} and \cite{crysrep}\footnote{Our notations here differ from the ones in Kisin's papers.}.\\
Write ${\bf A}^{[0,1)}=W[[u]]$ and ${\bf A}$ for the $p$-adic completion of $W((u))$. Further ${\bf B}={\bf A}[1/p]$.
Let $R$ denote Fontaine's ring 
\[R=\lim\limits_{\longleftarrow} \Ocal_{\mathbb{C}_p}/p\Ocal_{\mathbb{C}_p}\]
where the transition maps in the limit are given by the $p$-th power map. Let $\underline{\pi}$ denote the element $(\pi,\pi_1,\pi_2,\dots)\in R$
and write $[\underline{\pi}]$ for the Teichm\"uller representative of $\underline{\pi}$ in the Witt ring $W(R)$ of $R$.
We may regard ${\bf B}$ as a subring of $W(\Frac R)[1/p]$ by maping $u$ to $[\underline{\pi}]$. The lift $\phi$ of the absolute Frobenius on $R$  maps $[\underline{\pi}]$ to $[\underline{\pi}]^p$ and hence this embeddings is $\phi$-equivariant. Write $\widetilde{\bf B}$ for the closure of the maximal unramified extension of ${\bf B}$ inside $W(\Frac R)[1/p]$ and denote by $\widetilde{\bf A}$ its ring of integers. Then $\widetilde{\bf A}$ is a complete discrete valuation ring (for the $p$-adic topology) with residue field the closure of $k((u))^{\rm sep}$ in $\Frac R$ and we have an injection $\widetilde{\bf A}\hookrightarrow W(\Frac R)$ which is continuous for the canonical topology of $\widetilde{\bf A}$.\\
We write $\widetilde{\bf A}^{[0,1)}=\widetilde{\bf A}\cap W(R)\subset W(\Frac R)$. Then we have
\begin{align*}
\widetilde{\bf A}^{\phi=\id}&=\Z_p, & \widetilde{\bf A}^{G_{K_\infty}}&={\bf A},\\
\widetilde{\bf B}^{\phi=\id}&=\Q_p, & \widetilde{\bf B}^{G_{K_\infty}}&={\bf B}.
\end{align*}

The  Galois group $G_{K_\infty}$ is isomorphic to the absolute Galois group of $k((u))$, see \cite[Theorem 11.1.2]{survey} for example, and its representations on finite dimensional $\Q_p$-vector spaces are described in terms of \'etale $\phi$-modules over ${\bf B}$, i.e. finite dimensional ${\bf B}$-vector spaces $N$ together with an isomorphism $\Phi:\phi^\ast N\rightarrow N$
such that there exists an ${\bf A}$-lattice $\mathfrak{N}\subset N$ with $\Phi(\phi^\ast\mathfrak{N})=\mathfrak{N}$. Given an \'etale $\phi$-module $(N,\Phi)$ over ${\bf B}$, we write
\begin{equation}\label{GKinftyrepn}
V_{\bf B}(N)=(N\otimes_{\bf B}\widetilde{\bf B})^{\Phi=\id}
\end{equation}
for the associated $G_{K_\infty}$-representation.

\begin{rem}
In fact Kisin considers slightly smaller period rings. Instead of $\widetilde{\bf B}$ he considers the $p$-adic completion of the maximal unramified extension of ${\bf B}$ inside $W(\Frac R)[1/p]$. But as $\widetilde{\bf B}$ has the same $\phi$ and $G_{K_\infty}$-invariants the period ring considered in Kisin's papers, the theory works with our definitions as well. The advantage of our definition is that is makes it easier to define the sheafified versions of the period rings.
\end{rem}

Let $A_{\rm cris}$ denote the $p$-adic completion of the divided power envelope of $W(R)$ with respect to the kernel of the surjection $\theta:W(R)\rightarrow \Ocal_{\mathbb{C}_p}$ induced by 
\[[(x,x^{1/p},x^{1/p^2},\dots)]\longmapsto x.\]
Let $B_{\rm cris}^+=A_{\rm cris}[1/p]$ and $B_{\rm cris}=B_{\rm cris}^+[1/t]$ denote Fontaine's ring of crystalline periods, where
\[t=\log[(1,\epsilon_1,\epsilon_2,\dots)]\]
is the period of the cyclotomic character.

Recall that a representation of $G_K$ on a $d$-dimensional $\Q_p$-vector space $V$ is called \emph{crystalline} if 
\[D_{\rm cris}(V)=(V\otimes_{\Q_p}B_{\rm cris})^{G_K}\]
is of dimension $d$ over $K_0=B_{\rm cris}^{G_K}$. 
The $K_0$-vector space $D_{\rm cris}(V)$ is equipped with a semi-linear automorphism $\Phi$ and a descending, separated and exhaustive filtration $\Fcal^\bullet$ on $D_{\rm cris}(V)\otimes_{K_0}K$. Such an object is called a filtered $\phi$-module (or filtered isocrystal) over $K$.
A filtered isocrystal $(D,\Phi,\Fcal^\bullet)$ is called \emph{weakly admissible} if 
\begin{align*}
v_p(\det \Phi)&=\sum_i i\,\dim_K \gr_i(D\otimes_{K_0}K) \ , \\
v_p(\det \Phi|_{D'})&\geq \sum_i i\,\dim_K \gr_i (D'\otimes_{K_0}K)\ \ \forall\ D'\subset D,\ D'\ \Phi\text{-stable}.
\end{align*}
Then there is an equivalence of categories between the category of crystalline representations and the category of weakly admissible filtered $\phi$-modules over $K$, see \cite[Theorem A]{ColmezFont}.

\subsection{Adic spaces}
Let $A$ be a topological ring over $\Q_p$. By a \emph{valuation} (in the sense of \cite[2, Definition]{contval}) on $A$ we mean a map $v:A\rightarrow \Gamma_v\cup\{0\}$ to a totally ordered abelian group $\Gamma_v$ (written multiplicatively) such that
\begin{align*}
 v(0)&=0 \\
 v(1)&=1 \\
 v(ab)&=v(a)v(b)\\
 v(a+b)&\leq\max\{v(a),v(b)\},
\end{align*}
where the order on $\Gamma_v$ is extended to $\Gamma_v\cup\{0\}$ by $0<\gamma$ for all $\gamma\in\Gamma_v$.\\
The valuation is called continuous if $\{a\in A\mid v(a)\leq \gamma\}$ is open for all $\gamma\in\Gamma_v$.
If $A^+\subset A$ is an open and integrally closed subring, Huber defines the adic spectrum of $(A,A^+)$ as 
\[\Spa(A,A^+)=\left\{
\begin{array}{*{20}c}
\text{isomorphism classes of continuous valuations}\\ v:A\rightarrow \Gamma_v\cup\{0\}\ \text{such that}\ v(a)\leq 1\ \text{for all}\ a\in A^+
\end{array}\right\}.\]
This space is equipped with a structure sheaf $\Ocal_X$ and a sheaf of integral subrings $\Ocal_X^+$, see \cite[1]{Hu2}.

Given a finite extension $E$ of $\Q_p$ we denote by ${\Rig}_E$ the category of rigid analytic varieties over $E$ (cf. \cite[Part C]{BoschGR}) and by ${\Ad}_E$ the category of adic spaces over $E$ in the sense of Huber. We have a fully faithful embedding of ${\Rig}_E$ into ${\Ad}_E$ that factors through the full subcategory ${\Ad}_E^{\lft}$ of adic spaces locally of finite type over $E$, i.e. adic spaces that are locally isomorphic to $\Spa(A,A^\circ)$ for an $E$-algebra $A$ that is topologically of finite type over $E$ and where $A^\circ\subset A$ is the subring of power bounded elements. In fact, this embedding identifies the category of quasi-separated rigid spaces with the category of quasi-separated adic spaces locally of finite type over $E$ (see \cite[1.1.11]{Huber}). 

For an adic space $X$ and a point $x\in X$ we will write $\Ocal_{X,x}$ for the local ring at $x$ and $\mfrak_{X,x}$ for its maximal ideal. We denote the residue field at $x$ by $k(x)$ and write $v_x:k(x)\rightarrow \Gamma_x\cup\{0\}$ for the corresponding valuation. Further we write $\widehat{k(x)}$ for the completion of the residue field (with respect to its natural topology). 
We write
\[\iota_x:\Spa(\widehat{k(x)},\widehat{k(x)}^+)=\Spa(k(x),k(x)^+)\hookrightarrow X\]
for the inclusion of the point $x$.

If $\Fcal$ is a (quasi-coherent) sheaf of $\Ocal_X$-modules (resp. $\Ocal_X^+$-modules) on $X$ we write $\Fcal\otimes k(x)$ for the quotient of $\Fcal_x$ by $\mfrak_{X,x}\Fcal_x$ (resp. $\mfrak_{X,x}^+\Fcal_x$), where
\[\Fcal_x=\lim\limits_{\substack{\longrightarrow \\ U\ni x}}\Gamma(U,\Fcal).\]
More generally, let $X \mapsto \mathscr{R}_X$ be a sheaf of quasi-coherent algebras (over $\Ocal_X$ or $\Ocal_X^+$) such that for every map $f:X\rightarrow Y$ there is an induced map $f^{-1}\mathscr{R}_Y\rightarrow \mathscr{R}_X$. Let $\Fcal$ be a quasi-coherent $\mathscr{R}_Y$-module. Then we write 
\[f^\ast \Fcal=f^{-1}\Fcal\otimes_{f^{-1}\mathscr{R}_Y}\mathscr{R}_X.\]
If $\mathcal{G}$ is another (quasi-coherent) sheaf of $\mathscr{R}_Y$-modules, we write $\sheafHom_{\mathscr{R}_Y}(\Fcal,\Gcal)$ for the sheaf of \emph{continuous} homomorphisms from $\Fcal$ to $\mathcal{G}$.

We will talk about stacks on the category $\Ad^{\rm lft}_E$ and hence want to make clear what we mean by the fpqc-topology.
A morphism $f:X\rightarrow Y$ of adic spaces is flat if the local ring $\Ocal_{X,x}$ is flat over $\Ocal_{Y,f(x)}$ for all $x\in X$ or, equivalently, if for all open affinoids $U=\Spa(A,A^+)\subset Y$ and $V=\Spa(B,B^+)\subset f^{-1}(U)$, the ring $B$ is a flat $A$-algebra.
The morphism $f$ is called faithfully flat if it is flat and surjective. It is an fpqc-morphism if it is faithfully flat and quasi-compact.

In the case of rigid analytic varieties the notion of being faithfully flat is defined in terms of formal models (see \cite[3]{BoschGoertz} for example).
A quasi-compact morphism of formal $\Ocal_E$-schemes $f:\Xcal'\rightarrow \Xcal$ is called \emph{faithfully rig-flat} if there exist coverings $\Xcal=\bigcup_{i\in I}\Spf\,A_i$ of $\Xcal$ and $f^{-1}(\Spf\,A_i)=\bigcup_{j\in J_i}\Spf\,A'_j$ such that the induced map
\[\coprod_{j\in J_i}\Spec\, A'_j\longrightarrow \Spec\,A_i\]
is faithfully flat over the complement of the special fiber for all $i\in I$.\\
As usual a morphism is called quasi-compact if the pre-image of any quasi-compact subspace is quasi-compact.
The property of being faithfully rig-flat is stable under admissible blow-ups and hence gives rise to the notion of an fpqc-morphism in the category of rigid analytic varieties.\\
A morphism $f:X\rightarrow X'$ in ${\Ad}^{\lft}_E$ is an fpqc morphism if and only if it is quasi-compact and if locally on $X$ and $X'$ (for the topology of an adic space) it is induced by an fpqc morphism in ${\Rig}_E$.\\
By \cite[Theorem 3.1]{BoschGoertz} the category of coherent sheaves is an fpqc-stack on the category ${\Rig}_E$. Hence the same is true for the category ${\Ad}_E^{\lft}$, as we can glue coherent sheaves locally in the topology of an adic space.

Similarly to the case of stacks on the category of schemes we say that an fpqc-stack $F$ on the category ${\Rig}_E$ (resp. ${\Ad}_E^{\lft}$) is an \emph{Artin stack} if the diagonal is representable, quasi-compact and separated and if there exist a rigid space $U$ (resp. an adic space locally of finite type) and a smooth, surjective, relatively representable morphism $U\rightarrow F$.

\section{Filtered $\phi$-modules}

Throughout this section we denote by $F$ a topological field containing $\Q_p$ with a continuous valuation $v_F:F\rightarrow \Gamma_F\cup\{0\}$.
Recall that $K_0$ is an unramified extension of $\Q_p$ with residue field $k$ and write $f=[K_0:\Q_p]$. We write $\phi$ for the lift of the absolute Frobenius to $K_0$. Further $K$ is a totally ramified extension of $K_0$ with ramification index $e=[K:K_0]$. In this section we briefly introduce \emph{filtered isocrystals with coefficients}. We refer to \cite[2]{Hellmann} for a more detailed discussion.
\begin{defn}\label{defisoc}
An \emph{isocrystal over $k$ with coefficients in $F$} is a free $F\otimes_{\Q_p}K_0$-module $D$ of finite rank together with an automorphism $\Phi:D\rightarrow D$ that is semi-linear with respect to $\id\otimes\phi:F\otimes_{\Q_p}K_0\rightarrow F\otimes_{\Q_p}K_0$.\\
A morphism $f:(D,\Phi)\rightarrow (D',\Phi')$ is an $F\otimes_{\Q_p}K_0$-linear map $f:D\rightarrow D'$ such that
\[f\circ \Phi= \Phi'\circ f.\]
The category of isocrystals over $k$ with coefficients in $F$ is denoted by $\Isoc(k)_F$.
\end{defn}
\begin{defn}\label{deffilisoc}
 A \emph{$K$-filtered isocrystal over $k$ with coefficients in $F$} is a triple $(D,\Phi,\Fcal^\bullet)$, where $(D,\Phi)\in\Isoc(k)_F$ and $\Fcal^\bullet$ is a descending, separated and exhaustive $\Z$-filtration on $D_K=D\otimes_{K_0}K$ by (not necessarily free) $F\otimes_{\Q_p}K$-submodules. \\
A morphism 
\[f:(D,\Phi,\Fcal^\bullet)\longrightarrow (D',\Phi',\Fcal'^\bullet)\] 
is a morphism $f:(D,\Phi)\rightarrow (D',\Phi')$ in $\Isoc(k)_F$ such that $f\otimes\id:D_K\rightarrow D'_K$ respects the filtrations.\\
The category of $K$-filtered isocrystals over $k$ with coefficients in $F$ is denoted by $\Fil\Isoc(k)_F^K$.
\end{defn}
For an extension $F'$ of $F$  with valuation $v_{F'}:F'\rightarrow \Gamma_{F'}\cup\{0\}$ extending the valuation $v_F$ we have an \emph{extension of scalars}
\[-\otimes_FF':\Fil\Isoc(k)_F^K\longrightarrow \Fil\Isoc(k)_{F'}^K,\]
mapping $(D,\Phi,\Fcal^\bullet)\in\Fil\Isoc(k)_F^K$ to the triple $(D\otimes_FF',\Phi\otimes\id,\Fcal^\bullet\otimes_FF')$.
If $F'$ is finite over $F$, we also have a \emph{restriction of scalars}
\[\epsilon_{F'/F}:\Fil\Isoc(k)_{F'}^K\longrightarrow \Fil\Isoc(k)_F^K.\]
This functor maps $(D',\Phi',\Fcal'^\bullet)\in\Fil\Isoc(k)_{F'}^K$ to itself, forgetting the $F'$-action but keeping the $F$-action.
In the following we will often shorten our notation and just write $D$ for an object $(D,\Phi,\Fcal^\bullet)\in\Fil\Isoc(k)_F^K$.

We now define weakly admissible objects in $\Fil\Isoc(k)_F^K$.
Write $\Gamma_F\otimes\Q$ for the localisation of the abelian group $\Gamma_F$. Then every element $\gamma'\in\Gamma_F\otimes\Q$ can be written as a single tensor $\gamma\otimes r$ and we extend the total order of $\Gamma_F$ to $\Gamma_F\otimes\Q$ by 
\[a\otimes\tfrac{1}{m}<b\otimes\tfrac{1}{n}\Leftrightarrow a^n<b^m.\]
\begin{defn}
Let $(D,\Phi,\Fcal^\bullet)\in\Fil\Isoc(k)_F^K$. We define  
\begin{align*}
\deg\Fcal^\bullet&=\sum_{i\in\Z}\tfrac{1}{ef}i\dim_F\gr_i\Fcal^\bullet\\
\deg_F(D)&=(v_F(\dete_F\Phi^f)\otimes\tfrac{1}{f^2})^{-1}\ v_F(p)^{\deg(\Fcal^\bullet)}\in\Gamma_F\otimes\Q\\
\mu_F(D)&=\deg_F(D)(1\otimes\tfrac{1}{d})\in\Gamma_F\otimes\Q.
\end{align*}
We call $\mu_F(D)$ the \emph{slope} of $D$.
\end{defn}
\begin{rem}
One easily sees that the slope $\mu_F$ is preserved under extension and restriction of scalars. Hence we will just write $\mu$ in the sequel.
\end{rem}
\begin{defn}
An object $(D,\Phi,\Fcal^\bullet)\in\Fil\Isoc(k)_F^K$ is called \emph{weakly admissible} if, for all $\Phi$-stable subobjects $D'\subset D$, we have $\mu(D')\geq \mu(D)=1$.
\end{defn}
\begin{rem}
Let $(D,\Phi,\Fcal^\bullet)\in\Fil\Isoc(k)_{\Q_p}^K$ and denote for the moment by $v_p$ the $p$-adic valuation of $\Q_p$. Then 
\[\mu(D)=p^{v_p(\dete_{K_0}\Phi)-\sum_ii\dim_K(\Fcal^i/\Fcal^{i+1})}.\]
Hence we see that $D$ is weakly admissible if and only it is weakly admissible in the sense of \cite[3.4]{ColmezFont}. 
\end{rem}
There is a Harder-Narasimhan formalism for objects in $\Fil\Isoc(k)_F^K$ as in \cite[3]{FargFont} or \cite[Chapter 1]{DatOrlikRapo}. The main consequence of the existence of a Harder-Narasimhan filtration is that weak admissibility is preserved under extension and restriction of scalars. We refer to \cite[2]{Hellmann} for a more detailed discussion in the case of an arbitrary valuation.
\begin{prop}\label{waextstab}
Let $F'$ be an extension of $F$ with valuation $v_{F'}$ extending $v_F$ and $D\in\Fil\Isoc(k)_F^K$. 
If $D$ is semi-stable of slope $\mu$, then $D'=D\otimes_FF'$ is semi-stable of slope $\mu$.\\
If in addition $F'$ is finite over $F$ and $D'\in\Fil\Isoc(k)_{F'}^K$ is semi-stable of slope $\mu$, then $\epsilon_{F'/F}(D')\in\Fil\Isoc(k)_F^K$ is semi-stable of slope $\mu$. 
\end{prop}
\begin{proof}
This is \cite[Corollary 2.22]{Hellmann}.
\end{proof}

\section{Families of filtered $\phi$-modules}

We now consider families of the objects introduced in the last section parametrized by rigid analytic and adic spaces.
\subsection{Stacks of filtered $\phi$-modules}
Let $d$ be a positive integer and $\nu$ an algebraic cocharacter 
\begin{equation}\label{cocharnu}
\nu:\bar\Q_p^\times\longrightarrow (\Res_{K/\Q_p}A_K)(\bar\Q_p),
\end{equation}
where $A\subset \GL_d$ is the diagonal torus. We assume that this cocharacter is dominant with respect to the restriction $B$ of the Borel subgroup of upper triangular matrices in $(\GL_d)_K$. 
This cocharacter is defined over the reflex field $E\subset \bar\Q_p$. 
Let $\Delta$ denote the set of simple roots (defined over $\bar\Q_p$) of $\Res_{K/\Q_p}\GL_d$ with respect to $B$ and denote by $\Delta_\nu\subset \Delta$ the set of all simple roots $\alpha$ such that $\langle \alpha,\nu\rangle=0$. Here $\langle -,-\rangle$ is the canonical pairing between characters and cocharacters.
We write $P_\nu$ for the parabolic subgroup of $(\Res_{K/\Q_p}\GL_d)$ containing $B$ and corresponding to $\Delta_\nu\subset \Delta$. This parabolic subgroup is defined over $E$, and the quotient by this parabolic is a projective variety over $E$,
\begin{equation*}
\Gr_{K,\nu}=(\Res_{K/\Q_p}\GL_d)_E/P_\nu
\end{equation*}
representing the functor 
\[S\mapsto\{\text{filtrations}\ \Fcal^\bullet\ \text{of}\ \Ocal_S\otimes_{\Q_p}K^d\ \text{of type}\ \nu\}\]
on the category of $E$-schemes.  Here the filtrations are locally on $S$ direct summands. Being of \emph{type} $\nu$ means the following. Assume that the cocharacter 
\[\nu:\bar\Q_p^\times\longrightarrow \prod_{\psi:K\rightarrow \bar\Q_p}\GL_d(\bar\Q_p)\]
is given by cocharacters
\[\nu_\psi: \lambda\mapsto {\rm diag}((\lambda^{i_1(\psi)})^{(m_1(\psi))},\dots,(\lambda^{i_r(\psi)})^{(m_r(\psi))})\]
for some integers $i_j(\psi)\in\Z$ and multiplicities $m_j(\psi)>0$. Then any point $\Fcal^\bullet\in\Gr_{K,\nu}(\bar\Q_p)$ is a filtration $\prod_\psi\Fcal_\psi^\bullet$ of $\prod_\psi\bar\Q_p^d$ such that
\[\dim_{\bar\Q_p}\gr_i(\Fcal_\psi^\bullet)=\begin{cases}
0& \text{if}\ i\notin\{i_1(\psi),\dots,i_r(\psi))\}\\
m_j(\psi) & \text{if}\ i=i_j(\psi). 
\end{cases}\] 
We denote by $\Gr_{K,\nu}^{\rig}$ resp. $\Gr_{K,\nu}^{\ad}$ the associated rigid space, resp. the associated adic space (cf. \cite[9.3.4]{BoschGR}).

Given $\nu$ as in $(\ref{cocharnu})$ and denoting as before by $E$ the reflex field of $\nu$, we consider the following fpqc-stack $\Dfrak_\nu$ on the category ${\Rig}_E$ (resp. on the category ${\Ad}_E^{\lft}$).
For $X\in{\Rig}_E$ (resp. ${\Ad}_E^{\lft}$) the groupoid $\Dfrak_\nu(X)$ consists of triples $(D,\Phi,\Fcal^\bullet)$, where $D$ is a coherent $\Ocal_X\otimes_{\Q_p}K_0$-modules
which is locally on $X$ free over $\Ocal_X\otimes_{\Q_p}K_0$ and $\Phi:D\rightarrow D$ is an $\id\otimes\phi$-linear automorphism. Finally $\Fcal^\bullet$ is a filtration of $D_K=D\otimes_{\Q_p}K$ of type $\nu$, i.e. after choosing fpqc-locally on $X$ a basis of $D$, the filtration $\Fcal^\bullet$ induces a map to $\Gr_{K,\nu}^{\rig}$ (resp. $\Gr_{K,\nu}^{\ad}$), compare also \cite[5.a]{phimod}.\\ 

One easily sees that the stack $\Dfrak_\nu$ is the stack quotient of the rigid space
\begin{equation*}
X_\nu=(\Res_{K_0/\Q_p}\GL_d)_E^{\rig}\times\Gr_{K,\nu}^{\rig}
\end{equation*}
by the $\phi$-conjugation action of $(\Res_{K_0/\Q_p}\GL_d)_E^{\rig}$ given by 
\begin{equation*}
(A,\Fcal^\bullet).g=(g^{-1}A\phi(g),g^{-1}\Fcal^\bullet).
\end{equation*}
Here the canonical map
$X_\nu\rightarrow \Dfrak_\nu$
is given by 
\[(A,\Fcal^\bullet)\mapsto (\Ocal_{X_\nu}\otimes_{\Q_p}K_0^d, A(\id\otimes\phi),\Fcal^\bullet).\]
Hence $\Dfrak_\nu$ is an Artin stack in the sense of section $2$. Of course the corresponding statement in the category ${\Ad}^{\lft}_E$ is also true. 
\subsection{The weakly admissible locus}
Fix a cocharacter $\nu$ with reflex field $E$ as in the previous section.
If $X\in{\Ad}_E^{\lft}$ and $x\in X$, then our definitions imply that, given $(D,\Phi,\Fcal^\bullet)\in\Dfrak_\nu(X)$, we have 
\begin{align*}
(D\otimes k(x),\Phi\otimes\id,\Fcal^\bullet\otimes k(x)) & \in\Fil\Isoc(k)_{k(x)}^K\ \text{and}\\
\iota_x^\ast(D,\Phi,\Fcal^\bullet) &\in \Fil\Isoc(k)_{\widehat{k(x)}}^K.
\end{align*}
Our first main result is concerned with the structure of the weakly admissible locus in the stacks $\Dfrak_\nu$ defined above.
\begin{theo}\label{waopen}
Let $\nu$ be a cocharacter as in $(\ref{cocharnu})$ and $X$ be an adic space locally of finite type over the reflex field of $\nu$.
If $(D,\Phi,\Fcal^\bullet)\in\Dfrak_\nu(X)$, then the weakly admissible locus
\begin{align*}
X^{\rm wa}&=\{x\in X\mid (D\otimes k(x),\Phi\otimes\id,\Fcal^\bullet\otimes k(x))\ \text{is weakly admissible}\}\\
 &= \{x\in X\mid \iota_x^\ast(D,\Phi,\Fcal^\bullet)\ \text{is weakly admissible}\}
\end{align*}
is an open subset. Especially it has a canonical structure of an adic space.
\end{theo}
We can define a substack $\Dfrak_\nu^{\rm wa}\subset \Dfrak_\nu$ consisting of the weakly admissible filtered isocrystals. More precisely, for an adic space $X$ the groupoid $\Dfrak_\nu^{\rm wa}(X)$ consists of those triples $(D,\Phi,\Fcal^\bullet)$ such that $(D\otimes k(x),\Phi\otimes\id,\Fcal^\bullet\otimes k(x))$ is weakly admissible for all $x\in X$. Thanks to Proposition $\ref{waextstab}$ it is clear that this is again an fpqc-stack. The following result is now an obvious consequence of Theorem $\ref{waopen}$.
\begin{cor}
The stack $\Dfrak_\nu^{\rm wa}$ on the category of adic spaces locally of finite type over the reflex field of $\nu$ is an open substack of $\Dfrak_\nu$.
Especially it is again an Artin stack.
\end{cor}
\begin{proof}[Proof of Theorem $\ref{waopen}$]
First it is clear that 
\begin{align*}
&\{x\in X\mid (D\otimes k(x),\Phi\otimes\id,\Fcal^\bullet\otimes k(x))\ \text{is weakly admissible}\}\\
 = &\{x\in X\mid \iota_x^\ast(D,\Phi,\Fcal^\bullet)\ \text{is weakly admissible}\},
\end{align*}
as weak admissibility is stable under base change.
We only need to treat the case 
\[X=X_\nu=(\Res_{K_0/\Q_p}\GL_d)_E^{\ad}\times\Gr_{K,\nu}^{\ad}\]
which is locally (after choosing a basis of $D$) the universal case, again using the fact that weak admissibility is stable under base change.\\
Let $V=\Q_p^d$ and $V_0=K_0\otimes_{\Q_p}V$. Then $\Res_{K_0/\Q_p}\GL_d(R)=\GL(R\otimes_{\Q_p}V_0)$ for a $\Q_p$-algebra $R$.

For $i\in\{0,\dots,d\}$, consider the following functor on the category of $\Q_p$-schemes,
\begin{equation*}
 S\longmapsto
\left\{
{\begin{array}{*{20}c}
 \Ecal\subset \Ocal_S\otimes_{\Q_p}V_0\ \text{locally free}\ \Ocal_S\otimes_{\Q_p}K_0\text{-submodule} \\
\text{of rank}\ i\ \text{that is locally on}\ S\ \text{a direct summand}
\end{array}}
\right\}.
\end{equation*}
This functor is representable by a projective $\Q_p$-scheme $\Gr_{K_0,i}$.\\
We let $G=\Res_{K_0/\Q_p}\GL_d$ act on $\Gr_{K_0,i}$ in the following way: for a $\Q_p$-scheme $S$, let $A\in G(S)$ and $\Ecal\in\Gr_{K_0,i}(S)$. We get a linear endomorphism $A$ of $\Ocal_S\otimes_{\Q_p}V_0$ and define the action of $A$ on $\Ecal$ by
\[
 A\cdot\Ecal=A((\id\otimes\phi)(\Ecal)).
\]
Write
\[
a:G\times\Gr_{K_0,i}\longrightarrow \Gr_{K_0,i}
\]
for this action and consider the subscheme $Z_i\subset G\times\Gr_{K_0,i}$ defined by the following fiber product:
\begin{equation*}
\begin{aligned}
 \begin{xy}
  \xymatrix{
Z_i\ar[d]\ar[rr] && G\times\Gr_{K_0,i}\ar[d]^{a\times\id} \\
\Gr_{K_0,i}\ar[rr]^{\hspace{-5mm}\Delta} && \Gr_{K_0,i}\times\Gr_{K_0,i}.
}
 \end{xy}
\end{aligned}
\end{equation*}
An $S$-valued point $x$ of the scheme $Z_i$ is a pair $(g_x,U_x)$, where $g_x\in G(S)$ is a linear automorphism of $\Ocal_S\otimes_{\Q_p}V_0$ and $U_x$ is an $\Ocal_S\otimes_{\Q_p}K_0$-submodule of rank $i$ stable under $\Phi_x=g_x(\id\otimes\phi)$. The scheme $Z_i$ is projective over $G$ via the first projection
\[\pr_i:Z_i\longrightarrow G.\]
Further we denote by $f_i\in\Gamma(Z_i,\Ocal_{Z_i})$ the global section defined by 
\begin{equation*}
 f_i(g_x,U_x)=\dete ((g_x(\id\otimes\phi))^f|_{U_x})
\end{equation*}
(recall $f=[K_0:\Q_p]$).
We also write $f_i$ for the global section on the associated adic space $Z_i^{\ad}$.\\
We write $\Ecal$ for the pullback of the universal bundle on $Z_i$ to $Z_i\times \Gr_{K,\nu}$ and $\Fcal^\bullet$ for the pullback of the universal filtration on $\Gr_{K,\nu}$. Then the fiber product 
\[\Gcal^\bullet=(\Ecal\otimes_{K_0}K)\cap \Fcal^\bullet\]
is a filtration of $\Ecal\otimes_{K_0}K$ by coherent sheaves. By the semi-continuity theorem the function
\[h_i: x\longmapsto \sum_{i\in\Z}i\ \tfrac{1}{ef}\dim_{\kappa(x)}\gr_i\ \Gcal^\bullet\]
is upper semi-continuous on $Z_i\times \Gr_{K,\nu}$ and hence so is 
\[h_i^{\ad}: x\longmapsto \sum_{i\in\Z}i\ \tfrac{1}{ef}\dim_{k(x)}\gr_i\ (\Gcal^\bullet)^{\ad}.\]
For $m\in\Z$ we write $Y_{i,m}\subset Z_i\times \Gr_{K,\nu}$ (resp. $Y_{i,m}^{\ad}\subset Z_i^{\ad}\times\Gr_{K,\nu}^{\ad}$) for the closed subscheme (resp. the closed adic subspace) 
\begin{equation*}
\begin{aligned}
Y_{i,m}&= \{y\in Z_i\times \Gr_{K,\nu}\mid h_i(y)\geq m\},\\
Y_{i,m}^{\ad} &= \{y\in Z_i^{\ad}\times\Gr_{K,\nu}^{\ad}\mid h_i^{\ad}(y)\geq m\}.
\end{aligned}
\end{equation*} 
Then the definitions imply that
\[\pr_{i,m}:Y_{i,m}^{\ad}\longrightarrow X_\nu\]
is a proper morphism of adic spaces.\\
We write $X_0\subset X_\nu$ for the open subset of all $x=(g,\Fcal^\bullet)\in X_\nu$ such that
\[v_x(\dete(g(\id\otimes\phi)^f))\otimes \tfrac{1}{f^2}=v_x(p)^{-\sum_i\tfrac{1}{ef}i\dim\gr_i\Fcal^\bullet}.\]
Then $X_\nu^{\rm wa}\subset X_0$ and 
\begin{equation}\label{waadic}
\begin{aligned}
 &X_0\backslash X_\nu^{\rm wa}\\ 
 =&\bigcup_{i,m} {\rm pr}^{\rm ad}_{i,m}(\{y=(g_y,U_y,\Fcal_y^\bullet)\in Y^{\rm ad}_{i,m}\mid v_y(f_i)<v_y(p)^{-f^2\,m}\}),
\end{aligned}
\end{equation}
where the union on the right hand side runs over all $i\in\{1,\dots,d-1\}$ and $m\in \Z$, and one easily checks that this union is finite.\\
Indeed, let $x=(g_x,\Fcal^\bullet_x)\in G^{\rm ad}\times\Gr_{K,\nu}^{\rm ad}$. Then the object 
\[(k(x)\otimes V_0,g_x(\id\otimes\phi),\Fcal_x^\bullet)\]
is not weakly admissible if and only if there exists a $g_x(\id\otimes\phi)$-stable subspace $U_x\subset k(x)\otimes V_0$ of some rank, violating the weak admissibility condition.
This is exactly the condition on the right hand side of $(\ref{waadic})$. Here we implicitly use that fact that weak admissibility is stable under extension of scalars (see Proposition $\ref{waextstab}$).\\ 
Now the sets $\{y=(g_y,V_y,\Fcal_y^\bullet)\in Y^{\rm ad}_{i,m}\mid v_y(f_i)<v_y(p)^{-f^2\,m}\}$ are closed by the definition of the topology on an adic space.
Further the map ${\rm pr}_{i,m}$ is a proper map of adic spaces and hence universally closed. We conclude that the complement of the weakly admissible locus is closed.
\end{proof}
\begin{prop}\label{maxopen}
Let $X$ and $(D,\Phi,\Fcal^\bullet)$ as in Theorem $\ref{waopen}$. Then $X^{\rm wa}$ is the maximal open subset whose rigid analytic points are exactly the weakly admissible ones.
\end{prop}
\begin{proof}
First there is a maximal open subset $Y$ whose rigid analytic points are exactly the weakly admissible ones (i.e. the union over all open subsets with this property is non-empty), as $X^{\rm wa}$ is open.\\
Let $y\in Y$, then $(D\otimes k(y),\Phi\otimes\id,\Fcal^\bullet\otimes k(y))$ is of slope $1$, as there is an affinoid neighbourhood of $y$ such that the slope at all rigid analytic points is $1$ and hence by the maximum modulus principle the same is true at $y$.\\
As weak admissibility is preserved under base change we are again reduced to the case $X=X_\nu$. If $y$ is not weakly admissible then there exists a $\Phi_y$-stable subspace $U\subset k(y)^d\otimes_{\Q_p}K_0$ violating the weak admissibility condition. Then $y\in \pr_i(Z_i)^{\ad}\times\Gr_{K,\nu}$ for some $i$, where $\pr_i(Z_i)\subset G$ is a closed subscheme and hence $\pr_i(Z_i)^{\ad}\times\Gr_{K,\nu}$ is in ${\Ad}_E^{\lft}$. Now the $\Phi$-stable subobject $U$ is induced from a family $\mathcal{U}$ of $\Phi$-stable subobjects in an affinoid neighbourhood $\Spa(A,A^\circ)$ of $y$ inside $\pr_i(Z_i)^{\ad}\times\Gr_{K,\nu}$. 
But for all rigid analytic points of $\Spa(A,A^\circ)$ we must have
\[\mu(\mathcal{U}\otimes k(x),\Phi_x|_{\mathcal{U}\otimes k(x)},(\Fcal^\bullet\cap \mathcal{U}_K)\otimes k(x))\geq 1\]
and hence by the maximum principle the same must be true for $y$.
\end{proof}
\begin{expl}\label{noBerko}
We give an example which shows that the weakly admissible locus in a family of filtered $\phi$-modules is not in general an analytic space in the sense of Berkovich (cf. \cite{Berko}). \\
Let $K=K_0=\Q_p$ and denote by $\boldB$ the closed unit disc over $\Q_p$. Let $X=(\boldB\backslash \{0,x_0\})\times\boldB$, where $x_0$ is the point defined by $(T_1)^2-p=0$ , and where we denote by $T_1$ resp. $T_2$ the coordinate functions on the first (resp. the second) factor. Consider the family $(D,\Phi,\Fcal^\bullet)$, where $D=\Ocal_Xe_1\oplus\Ocal_Xe_2$ and $\Phi$ is the linear automorphism of $D$ given by the matrix
\[\begin{pmatrix} T_1 & 0 \\ 0 & pT_1^{-1}\end{pmatrix}.\]
The filtration $\Fcal^\bullet$ is given by
\[\Fcal^i=\begin{cases}D & \text{if}\ i\leq 0 \\ \Ocal_X(e_1+T_2e_2) & \text{if}\ i=1 \\ 0 & \text{if}\ i\geq 2.\end{cases}\]
Then one easily sees that the weakly admissible locus is given by
\[
\{x\in X\mid v_x(T_1)\geq v_x(p)\} \backslash  \{x\in X\mid T_2(x)=0\ \text{and}\ v_x(T_1)>v_x(p)\}.
\]
This is clearly an open adic subspace. \\
By \cite[8.3.1, 5.6.8]{Huber} we have a fully faithful embedding of the category of strict analytic spaces over $\Q_p$ in the category of adic spaces such that the essential image consists of the adic spaces $Y$ that are \emph{taut}, i.e. for every quasi-compact open subset $U\subset Y$ the closure $\overline{U}\subset Y$ is again quasi-compact.
However one easily sees that the closure of the quasi-compact open subset $\{x\in X\mid v_x(T_1)=v_x(p)\}$ inside $X^{\rm wa}$ is not quasi-compact \footnote{I thank R. Huber for pointing this out to me and hence shortening my original argument.}. 
\end{expl}

\section{Vector bundles on the open unit disc}

In this section we compare families of filtered $\phi$-modules with families of vector bundles over the open unit disc. 
For rigid analytic points this relation was established by Kisin in \cite{crysrep}, and is used in \cite{potsemdeform} to define crystalline (and more generally potentially semi-stable) deformation rings. Building on Kisin's construction, Pappas and Rapoport define a morphism from a stack of vector bundles on the open unit disc to the stack $\Dfrak_\nu$ of the previous sections that is a bijection at the level of classical points, see \cite[5]{phimod}. We recall the definition of this morphism and show that it is in fact an isomorphism.

For $0\leq r<1$ we write $\boldB_{[0,r]}$ for the closed unit disc of radius $r$ over $K_0$ in the category of adic spaces and denote by 
\[\Ubb=\lim\limits_{\substack{\longrightarrow \\ r\rightarrow 1}} \boldB_{[0,r]}\]
the open unit disc. This is an open subspace of the closed unit disc (which is \emph{not} identified with the set of all points $x$ in the closed unit disc with $|x|<1$ in the adic setting). 
In the following we will always write $u$ for the coordinate function on $\boldB_{[0,r]}$ and $\Ubb$, i.e. we view 
\[{\bf B}^{[0,r]}:=\Gamma(\boldB_{[0,r]},\Ocal_{\boldB_{[0,r]}})\]
 and ${\bf B}^{[0,1)}:=\Gamma(\Ubb,\Ocal_{\Ubb})$ as subrings of $K_0[[u]]$.

Let $X$ be an adic space over $\Q_p$ and write
\begin{align*}
\Bcal_X^{[0,r]}&=\Ocal_X\widehat{\otimes}_{\Q_p}{\bf B}^{[0,r]}={\rm pr}_{X,\ast}\Ocal_{X\times \boldB_{[0,r]}} \\
\Bcal_X^{[0,1)}&=\Ocal_X\widehat{\otimes}_{\Q_p}{\bf B}^{[0,1)}={\rm pr}_{X,\ast}\Ocal_{X\times \Ubb}
\end{align*}
for the sheafified versions of the rings ${\bf B}^{[0,r]}$ and ${\bf B}^{[0,1)}$. These are sheaves of topological $\Ocal_X$-algebras on $X$.

We consider a family of vector bundles on the open unit disc parametrised by $X$, i.e. a coherent sheaf $\Mcal$ on $X\times\Ubb$ that is fpqc-locally on $X$ free over $\Ocal_{X\times\Ubb}$. We equip this vector bundle with the following extra structure. Let $\phi:\Ubb\rightarrow \Ubb$ denote the morphism induced 
by the endomorphism of $K_0[[u]]$ that is the canonical Frobenius on $K_0$ and that maps $u$ to $u^p$. We again write $\phi$ for the map
\[\id\times\phi:X\times\Ubb\longrightarrow X\times\Ubb.\]
Then we equip $\Mcal$ with an injective homomorphism $\Phi: \phi^\ast\Mcal\rightarrow \Mcal$ such that $E(u)\coker\Phi=0$,
where $E(u)\in W[u]$ is the minimal polynomial of the fixed uniformizer $\pi\in\Ocal_K$ over $K_0$.\\
Equivalently, we can think of $\Mcal$ as a sheaf of $\Bcal_X^{[0,1)}$-modules that is locally on $X$ free over $\Bcal_X^{[0,1)}$ together with an endomorphism $\Phi:\phi^\ast\Mcal\rightarrow \Mcal$, as $X\times \Ubb\rightarrow X$ is Stein. Here we write $\phi$ for the endomorphism of $\Bcal_X^{[0,1)}$ that is induced from the Frobenius on ${\bf B}^{[0,1)}$.

We fix an integer $d>0$ and a dominant cocharacter $\nu$ with reflex field $E$ as in $(\ref{cocharnu})$. We will assume that $\nu$ satisfies the extra condition that the $S$-points of the associated flag variety $\Gr_{K,\nu}$ are filtrations $\Fcal^\bullet$ on $\Ocal_S\otimes K^d$ such that 
\begin{equation}\label{specialcochar}
\Fcal^i=\begin{cases} \Ocal_S\otimes_{\Q_p}K^d&\text{if}\ i\leq 0 \\ 0 &\text{if}\ i\geq 2.\end{cases}
\end{equation}
In particular $\nu$ is miniscule in this case.
Then we consider the following fpqc-stack $\Cfrak_\nu$ on the category of adic spaces locally of finite type over the reflex field of $\nu$:
For an adic space $X\in\Ad_E^{\rm lft}$ the groupoid $\Cfrak_\nu(X)$ consists of pairs $(\Mcal,\Phi)$ as above such that the reduction modulo $E(u)$ of the filtration 
\[E(u)\Phi(\phi^\ast\Mcal)\subset E(u)\Mcal\subset \Phi(\phi^\ast\Mcal)\]
is of type $\nu$ after locally choosing a basis.

For $n\geq 0$ we fix elements $r_n\in(|\pi|^{1/p^n}, |\pi|^{1/p^{n+1}})\cap p^\Q$ such that $r_n^{1/p}=r_{n+1}$ and  write $\boldB_n=\boldB_{r_n}$. Denote by $x_n$ the point in $\boldB_n$ corresponding to the irreducible polynomial $E(u^{p^n})$. We define 
\[\lambda_n=\prod_{i=0}^n \phi^i(E(u)/E(0))\in\Gamma(\boldB_n,\Ocal_{\boldB_n})={\bf B}^{[0,r_n]}.\]
 Then $\lambda_n$ vanishes precisely at the points $x_i\in\boldB_n$, for $i=1,\dots,n$.  Finally we write 
\begin{equation}\label{lambda}
\lambda=\lim\limits_{n\rightarrow\infty} \lambda_n=\prod_{n\geq 0}\phi^n(E(u)/E(0))\in{\bf B}^{[0,1)}.
\end{equation}
Let $X$ be an adic space locally of finite type over $E$ and $\Mcal$ a family of vector bundles over $\Ubb$ parametrised by $X$. We denote by $i:X\otimes_{\Q_p}K_0\rightarrow X\times\Ubb$ the inclusion $x\mapsto (x,0)$ and by $p:X\times\Ubb\rightarrow X\otimes_{\Q_p}K_0$ the projection. Given $\Mcal$, we set $D=i^\ast\Mcal$; then $\Phi:\phi^\ast\Mcal\rightarrow \Mcal$ induces a $\phi$-linear automorphism of $D$ which we again denote by $\Phi$.
\begin{lem}\label{lemxi}
There is a unique $\Phi$-compatible morphism of $\Ocal_X\otimes_{\Q_p}K_0$-modules $\xi:D\rightarrow p_\ast\Mcal$ such that\\
\noindent {\rm (i)} $i^\ast\xi$ is the identity on $D$\\
\noindent {\rm (ii)} The induced morphism $p^\ast\xi:p^\ast D\rightarrow \Mcal$ is injective and has cokernel killed by $\lambda$.\\
\noindent {\rm (iii)} If $r\in (|\pi|,|\pi|^{1/p})$ then the image of the restriction of $p^\ast\xi$ to $X\times\boldB_{[0,r]}$ coincides with the restriction of $\im(\Phi:\phi^\ast\Mcal\rightarrow \Mcal)$ to $X\times\boldB_r$.
\end{lem}
\begin{proof}
For affinoid adic spaces this follows from \cite[Proposition 5.2]{phimod}. The general case then follows by glueing on affinoid subdomains using the uniqueness.
\end{proof}
Using this lemma we can define a morphism 
\begin{equation}\label{mapD}
\begin{xy}
\xymatrix{
\Cfrak_\nu\ar[r]^{\underline{D}} & \Dfrak_\nu,
}
\end{xy}\end{equation} 
compare \cite[5.b]{phimod}.
Given $X$ and $(\Mcal,\Phi)\in\Cfrak_\nu(X)$ we define $(D,\Phi)$ as above. Further we have an isomorphism
\[\begin{xy}
\xymatrix{
D\otimes_{K_0}K\ar[r]^{\hspace{-2cm}\cong} & \Phi(\phi^\ast\Mcal)\otimes_{\Bcal_X^{[0,r]}}(\Ocal_X\otimes_{\Q_p}K)
}
\end{xy}
\]
for some $r\in p^{\Q}\cap (|\pi|,|\pi|^{1/p})$. This isomorphism is the reduction of $p^\ast\xi$ modulo $E(u)$. Now we set $\Fcal^1$ to be the pre-image of the subobject on the right hand side induced by $E(u)\Mcal\subset \im\,\Phi$. 

The construction of $\underline{D}(-)$ is obviously compatible with base change, i.e. if $f:X\rightarrow Y$ is a morphism in $\Ad_{\Q_p}^{\rm lft}$ and $(\Mcal,\Phi)\in\Cfrak_\nu(Y)$, then \[f^\ast\underline{D}(\Mcal)=\underline{D}(f^\ast\Mcal)\] and hence this construction defines a morphism of stacks.

We want to define a morphism 
\begin{equation}\label{mapM}
\begin{xy}
\xymatrix{
\Dfrak_\nu\ar[r]^{\underline{\Mcal}} & \Cfrak_\nu
}
\end{xy}\end{equation}
inverse to the map just defined. The idea is to generalise the construction in \cite[1.2]{crysrep} (resp. \cite[5.b.2]{phimod}).\\
Let $X=\Spa (A,A^\circ)$ be an affinoid adic space and $(D,\Phi,\Fcal^\bullet)\in\Dfrak_\nu(X)$. Then $p^\ast D=\Ocal_\Ubb\boxtimes D$ is a vector bundle on $X\times\Ubb$. We write $\widehat{\Ocal}_{\Ubb,x_n}=\widehat{\Ocal}_{\boldB_n,x_n}$ for the completed local ring at $x_n$. Then we have localizing maps
\[\Gamma(X\times\boldB_n,p^\ast D|_{X\times\boldB_n})\longrightarrow \widehat{\Ocal}_{\boldB_n,x_n}\otimes_{K_0}\Gamma(X,D)=\widehat{\Ocal}_{\boldB_n,x_n}\otimes_K \Gamma(X,D_K).\]
Precomposing this map with 
\[\sigma^{-n}\otimes\Phi^{-n}:\Gamma(X\times \boldB_n,p^\ast D)\longrightarrow \Gamma(X\times\boldB_n,p^\ast D),\]
where $\sigma$ is induced by the endomorphism of $K_0[[u]]$ that is the Frobenius on the coefficients, and inverting $\lambda_n$ we obtain a map 
\[i_n: {\bf B}^{[0,r_n]}[\lambda_n^{-1}]\otimes_{K_0} \Gamma(X,D)\longrightarrow \widehat{\Ocal}_{\boldB_n,x_n}[(u-x_n)^{-1}]\otimes_{K}\Gamma(X,D_K).\]
We define a coherent sheaf $\Mcal_n$ on $X\times\boldB_n$ by 
\[\Gamma(X,\Mcal_n)=\bigcap_{j=0}^n i_j^{-1}((u-x_j)^{-1} \widehat{\Ocal}_{\boldB_j,x_j}\otimes_K \Gamma(X,\Fcal^1)+ \widehat{\Ocal}_{\boldB_j,x_j}\otimes_{K}\Gamma(X,D_K)).\]
This is indeed a coherent sheaf as $X\times\boldB_n$ is affinoid with noetherian coordinate ring and $\Mcal_n\subset \lambda_n^{-1}p^\ast D|_{X\times\boldB_n}$.
Note that by construction we have
\[\Mcal_n|_{X\times\boldB_{n-1}}=\Mcal_{n-1}.\]
The morphism 
\[\id\otimes\Phi:\phi^\ast(\lambda^{-1}{\bf B}^{[0,1)}\otimes_{K_0}\Gamma(X,D))\longrightarrow {\bf B}^{[0,1)}[\lambda^{-1}]\otimes_{K_0}\Gamma(X,D)\]
induces morphisms of coherent sheaves 
\[\Phi_n:\phi^\ast\Mcal_{n-1}\longrightarrow \Mcal_n\]
that are injective and have cokernel killed by $E(u)$, i.e. the cokernel has support $X\times\{x_0\}$. This follows from the fact that, after reducing to the universal case $G\times\Gr_{K,\nu}$, we are dealing with coherent sheaves on a reduced space and that the corresponding assertions are true in the fibers over rigid analytic points of $X$ by the constructions in \cite{crysrep}.
Further we clearly have $\Phi_n|_{X\times\boldB_{n-1}}=\Phi_{n-1}$.\\
Using the coherent sheaves $\Mcal_n$ we may define a coherent sheaf $\Mcal$ on $X\times \Ubb$ by 
\[\Mcal|_{X\times\boldB_n}=\Mcal_n.\]
This sheaf is endowed with an injective morphism $\Phi:\phi^\ast\Mcal\rightarrow \Mcal$ with $E(u)\coker\,\Phi=0$ that is locally given by the $\Phi_n$ defined above. 
\begin{lem} The sheaf $\Mcal$ is fpqc-locally on $X$ free over $X\times\Ubb$.
\end{lem}
\begin{proof}
First we claim that there is a fpqc-covering $X'\rightarrow X$ such that all $\Mcal_n$ are free over $X'\times\boldB_n$.
This follows from the fact that we may replace $X$ locally in the fpqc-topology with open subsets of the universal space $X_\nu$ and hence with (a disjoint sum of) products of closed polydiscs with closed annuli. Now it suffices to see that $\Mcal_n$ is locally free, as all vector bundles on products of closed polydisc with closed annuli are trivial by Proposition $\ref{lokalfreifrei}$ below.
But $\Mcal_n$ is locally free since it is a coherent sheaf on a reduced space and it follows from the definition that $\Mcal_n\otimes k(x)$ has dimension $d$ for all $x\in X\times\boldB_n$.

Now we have to show that there is a compatible system of trivialisations of the $\Mcal_n$, i.e. there exist isomorphisms
\[\begin{xy}
\xymatrix{
\psi_n:\Ocal_{X'\times\boldB_n}^d\ar[r]^{\hspace{5mm}\cong}&\Mcal_n
}
\end{xy}\]
such that $\psi_n|_{X'\times\boldB_{n-1}}=\psi_{n-1}$. We write 
\[E_n={\rm Isom}(\Ocal_{X'\times\boldB_n}^d,\Mcal_n)\]
for the set of trivialisations of $\Mcal_n$ which is non empty by the above.\\ It follows that this space is a principal homogeneous space under the group $\GL_d(\Gamma(X'\times\boldB_n,\Ocal_{X\times\boldB_n}))$. The identity $\Mcal_n|_{X'\times\boldB_{n-1}}=\Mcal_{n-1}$ induces continuous injections $E_n\hookrightarrow E_{n-1}$.
As we assume that $X'$ is a product of closed polydiscs with closed annuli we find that $\GL_d(\Gamma(X'\times\boldB_n,\Ocal_{X\times\boldB_n}))$ is dense in $\GL_d(\Gamma(X'\times\boldB_{n-1},\Ocal_{X\times\boldB_{n-1}}))$  and hence $E_n$ is dense in $E_{n-1}$ under the above injections.
By the same arguments as in the proof of \cite[V. Theorem 2 and Proposition 2]{Gruson} we find that the intersection of the $E_n$ is non empty. 
An element of this intersection gives the compatible set of trivialisations that we were looking for. This shows that $\Mcal$ is free over $X'\times \Ubb$. 
\end{proof}
\begin{prop}\label{lokalfreifrei}
Let $E$ be a finite extension of $\Q_p$ and write $\boldB$ for the closed unit disc over $E$. Let $\partial\boldB=\{x\in\boldB \mid |x|=1\}$ and $X=\{x\in\boldB\mid |p|\leq |x|\leq 1\}$. Then every vector bundle on 
\[X^a\times\boldB^b\times\partial\boldB^c\]
is trivial for any $a,b,c\geq 0$.
\end{prop}
\begin{proof}
We proceed by induction on $a$. If $a=0$ then the claim follows (for all $b$ and $c$) from \cite[Satz 3]{Lutke}.

Assume the claim is true for $a$. We write
\[\boldB=X\cup_{\partial Y}Y,\]
where $Y=\{x\in\boldB\mid |x|\leq |p|\}\cong\boldB$ and $\partial Y=\{x\in\boldB\mid |x|=|p|\}\cong \partial\boldB$.
Note that $\boldB=X\cup Y$ is a covering of the closed unit disc by open subsets.\\
Let $\Ecal$ be a vector bundle on $X^{a+1}\times\boldB^b\times\partial\boldB^c$, then the restriction
\[\Ecal|_{\partial Y\times X^a\times\boldB^b\times\partial\boldB^c}\]
is trivial by induction hypothesis. Hence we may glue $\Ecal$ with the trivial bundle on $Y\times X^a\times\boldB^b\times\partial\boldB^c$ and obtain a vector bundle on $\boldB\times X^a\times\boldB^b\times\partial\boldB^c$ which must be trivial by induction hypothesis. It follows that $\Ecal$ must be trivial itself. 
\end{proof}
Again is is clear that the construction of $\underline{\Mcal}(-)$ is compatible with base change.
\begin{theo}
Let $\nu$ be a cocharacter as in $(\ref{cocharnu})$ satisfying $(\ref{specialcochar})$. The morphisms $\Cfrak_\nu\rightarrow \Dfrak_\nu$ and $\Dfrak_\nu\rightarrow \Cfrak_\nu$ defined in $(\ref{mapD})$ and $(\ref{mapM})$ are mutually inverse isomorphisms.
\end{theo}
\begin{proof}
We have to check that the morphisms are inverse to each other. \\
Given an adic space $X$ and an object $(D,\Phi,\Fcal^\bullet)\in\Dfrak_\nu(X)$, there is an obvious $\Phi$-compatible isomorphism
\[\Psi: D\longrightarrow \underline{D}(\underline{\Mcal}(D)).\]
We have to check that $\Psi$ respects the filtrations on both objects after extending scalars to $K$. But as the $\underline{\Mcal}$ and $\underline{D}$ are compatible with base change we may assume that $X$ is reduced, as there is a universal object over $\Res_{K_0/\Q_p}\GL_d\times \Gr_{K,\nu}$. Now for all rigid points $x\in X$ we have
\[\Psi(\Fcal^iD_K)\otimes k(x)= (\Fcal^i\underline{D}(\underline{\Mcal}(D))_K)\otimes k(x)\]
by \cite[Proposition 1.2.8]{crysrep}. As these are coherent sheaves this is enough to conclude that the map respects the filtrations.

Conversely, start with $\Mcal\in\Cfrak_\nu(X)$. Then the map $\xi$ of Lemma $\ref{lemxi}$ induces a map 
\begin{equation}\label{Mcalinclusion}\underline{\Mcal}(\underline{D}(\Mcal))\hookrightarrow \lambda^{-1} p^\ast D \rightarrow \lambda^{-1}\Mcal.
\end{equation}
We need to check that this induces an isomorphism $\underline{\Mcal}(\underline{D}(\Mcal))\rightarrow \Mcal$.\\
By Noether normalization we can locally on $X$ choose a finite morphism $X\rightarrow Y$ to a polydisc. Especially $Y$ is reduced and
we may view $\underline{\Mcal}(\underline{D}(\Mcal))$, $\Mcal$ and $\lambda^{-1}p^\ast D$ as coherent sheaves on $Y\times\Ubb$.
The morphism in $(\ref{Mcalinclusion})$ induces an isomorphism
\[\underline{\Mcal}(\underline{D}(\Mcal))\otimes k(y)\longrightarrow \Mcal\otimes k(y)\]
in the fibers over rigid analytic points of $Y$, as the claim is true for $\Q_p$, see \cite[Proposition 1.2.13]{crysrep}, and hence it is true for all finite $\Q_p$-algebras. Again this is enough to conclude that the map is an isomorphism
\end{proof}

\section{$\phi$-modules over the Robba ring}

In the previous section we have defined an isomorphism between stacks parametrising families of vector bundles over the open unit disc and certain filtered $\phi$-modules
generalising the construction in \cite{crysrep}. In loc.\,cit. Kisin shows that under the above isomorphism the condition of being weakly admissible translates to the condition of being purely of slope zero over the Robba ring (in the sense of \cite{Kedlaya}) for rigid analytic points. 
The idea is that the restriction of a vector bundle with semi-linear action of $\Phi$ on the open unit disc to some small annulus near the boundary still carries a lot of information about the vector bundle itself.\\
In fact, under the equivalence with crystalline representations, the restriction to small annuli near the boundary corresponds to the restriction of the Galois representation to $G_{K_\infty}$. Kisin shows that the restriction to $G_{K_\infty}$ is fully faithful on the category of crystalline representation, see \cite[Theorem 0.2]{crysrep}. For flat torsion representations (and hence for representations arising from $p$-divisible groups) this was shown by Breuil in \cite[Theorem 3.4.3]{Breuil}.\\
The slope theory for modules over a relative Robba ring was studied by Kedlaya and Liu in the rigid analytic setting in \cite{Liu} and \cite{KedlayaLiu},
where they already point out that the corresponding results fail in the category of analytic spaces in the sense of Berkovich (compare \cite[Remark 7.5]{KedlayaLiu}). However, we need to transpose their result to the category of adic spaces (locally of finite type) as this category seems to be the right category to formulate statements such as the local \'etaleness of $\phi$-modules.
\subsection{The relative Robba ring}
For $0\leq r\leq s<1$ we denote by $\boldB_{[r,s]}$ the annulus of inner radius $r$ and outer radius $s$ and by $\Ubb_{\geq r}$ the subspace of $\Ubb$ of inner radius $r$.\\
We write 
\begin{align*}
{\bf B}^{[r,s]}&=\Gamma(\boldB_{[r,s]},\Ocal_{\boldB_{[r,s]}}), & {\bf A}^{[r,s]}&=\Gamma(\boldB_{[r,s]},\Ocal^+_{\boldB_{[r,s]}}),\\
{\bf B}^{[r,1)}&=\Gamma(\Ubb_{\geq r},\Ocal_{\Ubb_{\geq r}}), &  {\bf A}^{[r,1)}&=\Gamma(\Ubb_{\geq r},\Ocal^+_{\Ubb_{\geq r}}),
\end{align*}
for the coordinate rings of these spaces and their integral subrings. 

Let $X$ be an adic space over $\Q_p$. We define the following sheaves of topological rings on $X$ which are relative versions of the rings introduced above:
For $0\leq r\leq s<1$ we define 
\begin{align*}
\Bcal_X^{[r,s]}&=\pr_{X,\ast}\Ocal_{X\times\boldB_{[r,s]}}= \Ocal_X\widehat{\otimes}_{\Q_p}{\bf B}^{[r,s]}, \\
\Acal_X^{[r,s]}&=\pr_{X,\ast}\Ocal^+_{X\times\boldB_{[r,s]}}= \Ocal^+_X\widehat{\otimes}_{\Z_p}{\bf A}^{[r,s]},\\
\Bcal_X^{[r,1)}&=\pr_{X,\ast}\Ocal_{X\times\Ubb_{\geq r}}= \Ocal_X\widehat{\otimes}_{\Q_p}{\bf B}^{[r,1)}, \\
\Acal_X^{[r,1)}&=\pr_{X,\ast}\Ocal^+_{X\times\Ubb_{\geq r}}= \Ocal^+_X\widehat{\otimes}_{\Z_p}{\bf A}^{[r,1)},
\end{align*}
where $\pr_X$ denotes the projection to $X$ and by the completed tensor product we mean the sheafification of the competed tensor product of presheaves of topological rings. 

We now define the relative Robba ring and the ring of integral elements in the Robba ring to be the sheaves of rings 
\begin{align*}
\Bcal_X^R&=\lim\limits_{\substack{\longrightarrow \\ r\rightarrow 1}} \Bcal_X^{[r,1)}\\
\Acal_X^\dagger &=\lim\limits_{\substack{\longrightarrow \\ r\rightarrow 1}} \Acal_X^{[r,1)}
\end{align*}
on $X$. Again these are sheaves of topological rings.
Then for a quasi-compact open subset $U$ of $X$ we have 
\begin{align*}
\Gamma(U,\Bcal_X^R)&=\lim\limits_{\substack{\longrightarrow \\ r\rightarrow 1}} \Gamma(U,\Bcal_X^{[r,1)})\\
\Gamma(U,\Acal_X^\dagger)&=\lim\limits_{\substack{\longrightarrow \\ r\rightarrow 1}} \Gamma(U,\Acal_X^{[r,1)}).
\end{align*}
We denote by 
\[\phi:\boldB_{[r^{1/p},s^{1/p}]}\longrightarrow \boldB_{[r,s]}\]
the morphism induced by the map which is the Frobenius on $K_0$ and which maps $u$ to $u^p$. This induces maps
\[\Ubb_{\geq r^{1/p}}\longrightarrow \Ubb_{\geq r}\]
which we again denote by $\phi$. If $X=\Spa\,(A,A^+)$ is an affinoid space we also write $\phi$ for the induced morphisms in the coordinate rings
\begin{align*}
\phi:A\widehat{\otimes}_{\Q_p}{\bf B}^{[r,s]}&=\Gamma(X\times\boldB_{[r,s]},\Ocal_{X\times\boldB_{[r,s]}})\longrightarrow A\widehat{\otimes}_{\Q_p}{\bf B}^{[r^{1/p},s^{1/p}]}\\
\phi:A\widehat{\otimes}_{\Q_p}{\bf B}^{[r,1)}&=\Gamma(X\times\Ubb_{\geq r},\Ocal_{X\times\Ubb_{\geq r}})\longrightarrow A\widehat{\otimes}_{\Q_p}{\bf B}^{[r^{1/p},1)}.
\end{align*}
In the limit these morphisms also define a morphism $\phi$ on $\Bcal_X^R$ that preserves the integral subring.

If $F$ is a (complete) topological field over $\Q_p$ with continuous valuation $v_F$ and integral subring $F^+$, we simply write $\Bcal_F^{[r,s]}=\Bcal_{\Spa(F,F^+)}^{[r,s]}$ and similarly for the other sheaves of rings defined above.

Now we can define families of $\phi$-modules over the Robba ring parametrized by an adic space $X$. In our context this is the kind of object that appears if we restrict the families of vector bundles on the open unit disc to small annuli near the boundary of the open unit disc.
\begin{defn}\noindent (i) A \emph{family of $\phi$-modules over the Robba ring  parametrised by $X$} is an $\Bcal_X^R$-module $\Ncal$ that is locally on $X$ free of finite rank over $\Bcal_X^R$ together with an isomorphism $\Phi:\phi^\ast\Ncal\rightarrow \Ncal$.\\
\noindent (ii) A family of $\phi$-modules of rank $d$ over $\Bcal_X^R$ is called \emph{\'etale} if there exists a covering $(U_i)_{i\in I}$ of $X$ and free $\Acal_{U_i}^\dagger$-submodules $N_i$ of $\Ncal|_{U_i}$ of rank $d$ such that
\[
 N_i\otimes_{\Acal_{U_i}^\dagger}\Bcal_{U_i}^R=\Ncal|_{U_i}
\]
and $\Phi$ induces an isomorphism $\phi^\ast N_i\rightarrow N_i$.\\
\noindent (iii) A family $\Ncal$ of $\phi$-modules over the Robba ring parametrized by $X$ is called \emph{\'etale at} $x\in X$ if the $\phi$-module 
$\iota_x^\ast\Ncal$ over $\Bcal_{\widehat{k(x)}}^R$ is \'etale.
\end{defn}
\begin{defn}\label{phimodonannulus}
\noindent (i) Let $0<r<r^{1/p}\leq s<1$. A $\phi$\emph{-module over} $\Bcal_X^{[r,s]}$ is an $\Bcal_X^{[r,s]}$-module $\Ncal$ that is locally on $X$ free of finite rank together with an isomorphism
\[\Phi:\phi^\ast\Ncal\otimes_{\Bcal_X^{[r^{1/p},s^{1/p}]}}{\Bcal_X^{[r^{1/p},s]}}\longrightarrow \Ncal\otimes_{\Bcal_X^{[r,s]}}{\Bcal_X^{[r^{1/p},s]}}.\]
\noindent (ii) A $\phi$-module $(\Ncal,\Phi)$ over $\Bcal_X^{[r,s]}$ is called \emph{r-\'etale}, if there exists locally on $X$ a free $\Acal_X^{[r,s]}$-submodule $N$ such that $\Ncal=N\otimes_{\Acal_X^{[r,s]}}\Bcal_X^{[r,s]}$ and 
\[\Phi(\phi^\ast N\otimes_{\Acal_X^{[r^{1/p},s^{1/p}]}}\Acal_X^{[r^{1/p},s]})=N\otimes_{\Acal_X^{[r,s]}}\Acal_X^{[r^{1/p},s]}.\] 
\end{defn}

\begin{lem}\label{rigptRobba}
Let $X\in \Ad_{\Q_p}^{\rm lft}$ and $x\in X$.\\
\noindent {\rm (i)} If $x$ is a rigid analytic point of $X$, i.e. $k(x)$ is a finite extension of $\Q_p$, then $\Bcal_X^R\otimes k(x)$ is identified with the classical Robba ring of $k(x)$ in the sense of \cite[Definition 1.1.1]{Kedlaya}, i.e.
\[\Bcal_{X}^R\otimes k(x)=\lim\limits_{\substack{\longrightarrow \\ r\rightarrow 1}} \Gamma(\Ubb_{\geq r}\otimes_{\Q_p}k(x), \Ocal_{\Ubb_{\geq r}\otimes_{\Q_p}k(x)})=\Bcal_{\widehat{k(x)}}^R.\]
\noindent {\rm (ii)} If $x$ is an arbitrary point, then 
\[\Bcal_X^R\otimes k(x)\subset \Bcal_{\widehat{k(x)}}^R\]
is a dense subring.
\end{lem}
\begin{proof}
\noindent (i) Our definitions imply that there is a canonical injection
\[\Bcal_X^R\otimes k(x)\hookrightarrow\lim\limits_{\substack{\longrightarrow \\ r\rightarrow 1}} \Gamma(\Ubb_{\geq r}\otimes_{\Q_p}k(x), \Ocal_{\Ubb_{\geq r}\otimes_{\Q_p}k(x)}).\]
Let $\sum_{i\in\Z}\bar f_i u^i$ be an element on the right hand side. As $k(x)$ is finite over $\Q_p$ there is a neighbourhood $U$ of $x$ in $X$ such that all functions of $k(x)$ lift to $U$. Hence there are functions $f_i\in\Gamma(U,\Ocal_X)$ that reduce to $\bar f_i$ in $k(x)$.\\
\noindent (ii) This is an direct consequence of the definitions and the easy point, that 
\[\Bcal_X^{[r,s]}\otimes k(x)\subset \Bcal_{\widehat{k(x)}}^{[r,s]}\]
is dense. 
\end{proof}
\begin{rem}
The obvious analogue of this Lemma for the integral ring $\Acal_{\widehat{k(x)}}^\dagger$ is also true, i.e. 
\[\Acal_X^\dagger\otimes k(x)=\lim\limits_{\substack{\longrightarrow \\ r\rightarrow 1}} \Gamma(\Ubb_{\geq r}\otimes_{\Q_p}k(x), \Ocal^+_{\Ubb_{\geq r}\otimes_{\Q_p}k(x)})=\Acal_{\widehat{k(x)}}^\dagger,\]
if $x\in X$ is a rigid point. Further 
\[\Acal_X^\dagger\otimes k(x)\subset \Acal_{\widehat{k(x)}}^\dagger\]
is dense if $x$ is arbitrary.
\end{rem}

For the rest of this section we will use the following notation:\\
Let $r\in p^\Q\cap[0,1)$. Then we write $r_i=r^{1/p^i}$.
We end this subsection by proving that it is enough to work over some closed annulus near the boundary in order to show that a family of $\phi$-modules over the Robba ring is \'etale.
We say that a family $(\Ncal,\Phi)$ of $\phi$-modules over the Robba ring parametrised by $X$ \emph{admits a model} $\Ncal_r$ \emph{over} $\Ubb_{\geq r}$ if there exist a vector bundle $\Ncal_r$ over $X\times\Ubb_{\geq r}$ (that is locally on $X$ free) and an isomorphism 
\[\Phi:\phi^\ast\Ncal_r\longrightarrow \Ncal_r|_{X\times\Ubb_{\geq r_1}}\]
such that 
\[\Ncal_r\otimes_{\Bcal_X^{[r,1)}}\Bcal_X^R=\Ncal\]
compatible with the $\phi$-linear maps on both sides. 
\begin{prop}\label{etaleonannulus}
Let $X=\Spa(A,A^+)\in\Ad_{\Q_p}^{\rm lft}$ and $(\Ncal,\Phi)$  be a free $\phi$-module over $\Bcal_X^R$ that admits a model $\Ncal_r$ over $\Ubb_{\geq r}$.
Then $\Ncal$ is \'etale if $\Ncal_r$ is $r$-\'etale in the sense of Definition $\ref{phimodonannulus}$.
\end{prop}
\begin{proof}
Denote by $\Xcal=\Spf\,A^+$ the formal model for $X$, i.e. $X=\Xcal^{\rm ad}$.
Further let us write 
\[\boldB^{(i)}=\boldB_{[r,r^{1/p^i}]}\]
and 
\[\partial\boldB_i=\partial\boldB_{[0,r^{1/p^i}]}=\boldB_{[r^{1/p^i},r^{1/p^i}]}\]
for the moment.

We now inductively construct lattices $N^{(i)}$ in the restriction of $\Ncal_r$ to $X\times\boldB^{(i)}$ that are stable under $\Phi$.
We start with $N^{(2)}=N_r$ on $X\times\boldB^{(2)}$. Then we define $N^{(i+1)}$ by glueing $N^{(i)}$ and $\phi^\ast N^{(i)}$ over $X\times\boldB_{[r_1,r_i]}$, using the isomorphism $\Phi$. This yields a sheaf $N$ on $X\times\Ubb_{\geq r}$. We claim that $N$ is free over $\Acal_X^{[r,1)}$.

We may choose 
\[\mathcal{Y}_{i+1}=\mathcal{Y}_i\cup \mathcal{Z}_i\]
as a formal model for $\boldB^{(i)}$, where $\mathcal{Y}_1$ is the canonical formal model for $\boldB^{(1)}$ that becomes isomorphic to
\[\widehat{\Abb}^1\cup\widehat{\Abb}^1\]
after a ramified extension of $K_0$ and the zero section of the first formal affine line is glued together with the zero section of the second formal affine line.
Further $\mathcal{Z}_i$ is the canonical formal model for $\boldB_{[r_{i-1},r_i]}$. Then the special fiber of $\mathcal{Y}_i$ is the $k$-scheme
\[\bar{\mathcal{Y}}_i=\Abb_k^1\cup\Pbb_k^1\cup\dots\cup \Pbb_k^1\cup \Abb_k^1\]
where the infinity section of the $j$-th $\Pbb_k^1$ is glued to the $0$-section of the $j+1$-th $\Pbb_k^1$ and the infinity section of the $i-1$-th (i.e. the last) $\Pbb_k^1$ is glued to the $0$-section of $\Abb_k^1$.  In the limit we find that $\lim\limits_{\longrightarrow}\mathcal{Y}_i$ is a formal model for $\Ubb_{\geq r}$. 

Further $\Xcal\times\mathcal{Y}_i$ is a formal model for $X\times\boldB^{(i)}$ and $\lim\limits_{\longrightarrow}\Xcal\times\mathcal{Y}_i$ is a formal model for $X\times\Ubb_{\geq r}$.

Our definition of the sheaf $N_i$ also defines vector bundles on $\Xcal\times\mathcal{Y}_i$ that are trivial on the product of $\Xcal$ with each irreducible component of $\mathcal{Y}_i$. Hence we can find trivialisations of these vector bundles and even compatible ones as $i$ grows. The analytification of this compatible trivialisation on the formal model yields a trivialisation
\[\Ocal_{X\times\Ubb_{\geq r}}^d\longrightarrow \Ncal\]
mapping $(\Ocal_{X\times\Ubb_{\geq r}}^+)^d$ to the sheaf $N$. This yields the claim.
\end{proof}

\subsection{Local \'etaleness}
In this section we study how the property of being \'etale behaves in families. 
Given a family of $\phi$-modules over the Robba ring we will show that the subset where the $\phi$-module is \'etale is an open subset. This will imply that a family of $\phi$-modules is \'etale if and only if all fibers are \'etale.

We use the following notation. Let $X$ be a quasi-compact adic space locally of finite type over $\Q_p$ and $f\in\Gamma(X,\Ocal_X)$. We write
\[|f|_X=\max_{x\in X(\bar\Q_p)}\{v_x(f)\}\in\mathbb{R},\]
where the maximum is taken over all rigid analytic points of $X$. By the maximum modulus principle we then also have
\[v_x(f)\leq |f|_X\]
for all points $x\in X$ which makes sense as every point is the secondary specialisation of a rank $1$ valuation (cf. \cite[Lemma 1.1.10]{Huber}).
\begin{lem}\label{bouningLemma}
Let $X=\Spa(A,A^+)$ be an affinoid adic space and $x\in X$. For any $\lambda>0$ there is an affinoid subdomain $Y=\Spa(B,B^+)$ of $X$ such that for all $f\in A$ vanishing at $x$ we have $v_y(f)\leq\lambda|f|_X$ for all $y\in Y$.
\end{lem}
\begin{proof}
Denote by $I=(f_1,\dots,f_r)$ the ideal of functions vanishing at $x$. Multiplying with an arbitrary power of $p$ we may assume that $|f_i|_X=1$. Hence the affinoid subdomain defined by $|f_i(y)|\leq \lambda$ for all $i$ satisfies the required properties.
\end{proof}
Let $X=\Spa(A,A^+)$ be an affinoid adic space. For $0\leq r\leq s<1$ we write
\begin{align*}
{\bf B}_A^{[r,s]}&:=\Gamma(X\times \boldB_{[r,s]},\Ocal_{X\times\boldB_{[r,s]}})&&\hspace{-2.5cm}=A\widehat{\otimes}_{\Q_p}{\bf B}^{[r,s]}, \\
{\bf A}_A^{[r,s]}&:=\Gamma(X\times \boldB_{[r,s]},\Ocal^+_{X\times\boldB_{[r,s]}})&&\hspace{-2.5cm}=A^+\widehat{\otimes}_{\Z_p}{\bf A}^{[r,s]}.
\end{align*}
\begin{defn}
Let $X=\Spa\,(A,A^+)$ be an affinoid adic space and fix $0<r_1\leq s\leq r_2<1$. We define a valuation $w_s$ on ${\bf B}_A^{[r_1,r_2]}$ as follows. Given 
\[f=\sum_i f_i u^i\in{\bf B}_A^{[r_1,r_2]}\]
we define
\[w_s(f)=\sup_i p^{-i}|f_i|_X^{\widetilde{s}},\]
where $s=p^{-\widetilde{s}}$.
This is just the multiplicative version of the valuation defined in \cite[Definition 7.2]{KedlayaLiu}.
\end{defn}
In the following we will use the following notation: If $D$ is a matrix and $v$ a valuation, we write $v(D)$ for the maximum of the valuations of the coefficients of $D$.

\begin{lem}\label{approxLemma}
Let $X=\Spa\,(A,A^+)$ be an affinoid adic space and let $D\in \GL_d({\bf B}_A^{[s,s]})$. We set 
\[h=(w_s(D)w_s(D^{-1}))^{-1}.\] 
Let $F\in {\rm Mat}_{d\times d}({\bf B}_A^{[s,s]})$ such that 
\[w_s(FD^{-1}-I_d)\leq c\,h^{1/{p-1}}\]
for some $c<1$. If $s^{2(p-1)}\geq c^p$, then there exists $U\in \GL_d({\bf B}_A^{[s,s^{1/p}]})$ such that 
\[W=U^{-1}F\phi(U)D^{-1}-I_d\in {\rm Mat}_{d\times d}(p{\bf A}_A^{[s,1)})\] and  
$w_s(W)<1$.
\end{lem}
\begin{proof}
This is \cite[Lemma 7.3]{KedlayaLiu}.
\end{proof}
\begin{theo}\label{loket}
Let $X$ be an adic space locally of finite type over $\Q_p$ and $\Ncal$ be a family of free $\phi$-modules of rank $d$ over $\Bcal_X^{[r,r_2]}$. Assume that there exists $x\in X$ and an $\Acal_X^{[r,r_2]}\otimes k(x)$-lattice $N_x\subset \Ncal\otimes k(x)$ such that $\Phi$ induces an isomorphism
\begin{equation}\label{etonannulus1}
\Phi:\phi^\ast(N\otimes_{\Acal_X^{[r,r_2]}}\Acal_X^{[r,r_1]})\longrightarrow N\otimes_{\Acal_X^{[r,r_2]}}\Acal_X^{[r_1,r_2]}.
\end{equation}
Then there exists an open neighbourhood $U\subset X$ of $x$ and a locally free $\Acal_U^{[r,r_2]}$-submodule $N\subset \Ncal$ of rank $d$ such that 
\begin{align*}
N\otimes k(x) &= N_x\\ 
\Phi(\phi^\ast N|_{X\times \boldB_{[r,r_1]}})&=N|_{X\times \boldB_{[r_1,r_2]}}\\
N\otimes_{\Acal_U^{[r,r_2]}}\Bcal_U^{[r,r_2]}&=\Ncal|_U.
\end{align*}
\end{theo}
\begin{proof}
In this formulation the proof is the same as the proof of \cite[Theorem 7.4]{KedlayaLiu}. Let us sketch the argument.

We may assume that $X=\Spa(A,A^+)$ is affinoid.
Choose a basis $\underline{e}_x$\ of $N_x$ and denote by $D_0\in \GL_d(\Acal_X^{[r,r_2]}\otimes k(x))$ the matrix of $\Phi$ in this basis. After shrinking $X$ if necessary we may lift the matrix $D$ to a matrix with coefficients in $\Gamma(X,\Acal_X^{[r,r_2]})$. Localizing further we may assume that $D$ is invertible over $\Gamma(X,\Acal_X^{[r,r_2]})$.

Fixing a basis $\underline{b}$ of $\Ncal$ we denote by $S\in \GL_d({\bf B}_A^{[r,r]})$ the matrix of $\Phi$ in this basis. Further we denote by $V$ a lift of the change of basis matrix from the basis $\underline{e}_x$ to the basis $\underline{b}\mod x$.\\
Applying Lemma $\ref{bouningLemma}$ we can shrink $X$ such that the valuation of the matrix $V^{-1}S\phi(V)-D$ is arbitrary small, as the valuation vanishes at $x$.
Now we can apply Lemma $\ref{approxLemma}$ with the matrix $F=V^{-1}S\phi(V)$ and obtain a matrix 
\[U\in \GL_d({\bf B}_A^{[r,r_1]}).\]
Then the basis of $\Ncal$ obtained from $\underline{b}$ by base change via the matrix $VU$ defines the desired integral model $N$ on $X\times\boldB_{[r,r_1]}$.
Gluing $N$ and $\phi^\ast N$ on $X\times \boldB_{[r_1,r_1]}$ using $\Phi$ we obtain the desired model. 
\end{proof}

\begin{prop}\label{latticedecompletion}
Let $X\in\Ad^{\rm lft}_{\Q_p}$ and $\mathcal{N}$ be a $\phi$-module over $\Bcal_X^{[r,r_2]}$. Let $x\in X$ such that $\iota_x^\ast\mathcal{N}$ is r-\'etale in the sense of Definition $\ref{phimodonannulus}$. Then $\mathcal{N}\otimes k(x)$ admits an $\Acal_X^{[r,r_2]}$-lattice that is $\Phi$-stable in the sense of $(\ref{etonannulus1})$.
\end{prop}
\begin{proof}
So simplify the notations, we write $F=k(x)$ and $E=\widehat{k(x)}$. Further we write $\mathfrak{m}\subset E^+$ for the maximal ideal of elements with valuation strict less than one.\\
Let $N$ be an $\Acal_{E}^{[r,r_2]}$-lattice in $\iota_x^\ast\mathcal{N}$ such that 
\begin{equation}\label{etonannulus2}
\Phi(\phi^\ast N|_{\{x\}\times \boldB_{[r,r_1]}})=N|_{\{x\}\times \boldB_{[r_1,r_2]}}.
\end{equation}
Consider the $\Acal_X^{[r,r_2]}\otimes k(x)$-module $\mathfrak{N}=N\cap \mathcal{N}\otimes k(x)\subset \iota_x^\ast\mathcal{N}$.
We claim that the canonical map
\[\mathfrak{N}/(\mathfrak{m}N \cap\mathcal{N}\otimes k(x))\longrightarrow N/\mathfrak{m}N\]
is an isomorphism: Indeed it is injective by definition and it remains to check surjectivity.
But $\mathcal{N}\otimes k(x)$ is dense in $\iota_x^\ast\mathcal{N}$ and $N\backslash \mathfrak{m}N\subset \iota_x^\ast\Ncal$ is open.
Hence we can approximate a basis of $N$ by elements in $\mathfrak{N}$. As $\mathfrak{m}$ contains a fundamental system of neighbourhoods of $0$, the surjectivity follows.

Now we can lift a basis of $N/\mathfrak{m}N$ to elements $f_1,\dots,f_d\in \mathfrak{N}$. Consider the (free) $\Acal_X^{[r,r_2]}\otimes k(x)$-submodule $\mathfrak{N}'$ of $\mathcal{N}\otimes k(x)$ that is generated by $f_1,\dots,f_d$.

By Nakayama's Lemma $f_1,\dots,f_d$ generate $N$ over $\Acal_E^{[r,r_2]}$ and hence also $\iota_x^\ast\mathcal{N}$ over $\Bcal_E^{[r,r_2]}$.
Now the maximal ideals $\mathfrak{m}_a$ of $\Bcal_X^{[r,r_2]}\otimes k(x)$ are parametrized by the Galois orbits in $\{a\in \bar F\mid r\leq v_F(a)\leq r_2\}$.
It follows that the images of $f_1,\dots,f_d$ generate $\mathcal{N}\otimes k(x)$ modulo each maximal ideal $\mathfrak{m}_a$,
as they generate $\iota_x^\ast\Ncal$ modulo the closure of $\mathfrak{m}_a$ in $\Bcal_E^{[r,r_2]}$. Hence $f_1,\dots,f_d$ generate $\mathcal{N}\otimes k(x)$ by Nakayama's Lemma.

If we denote by $B$ the matrix of $\Phi$ in the basis $f_1,\dots,f_d$, then the above implies
\[B\in\GL_d(\Acal_E^{[r,s]})\cap\GL_d(\Bcal_X^{[r,s]}\otimes k(x))=\GL_d(\Acal_X^{[r,s]}\otimes k(x)).\]
This finishes the proof of the Proposition.
\end{proof}

\begin{cor}\label{etaleopen}
Let $X$ be an adic space locally of finite type over $\Q_p$ and $(\Ncal,\Phi)$ a family of $\phi$-modules over $\Bcal_X^R$.\\
\noindent {\rm (i)} Let $x\in X$. Then $\iota_x^\ast\Ncal$ is \'etale if and only if $\Ncal\otimes k(x)$ is \'etale, i.e. if $\Ncal\otimes k(x)$ admits a $\Phi$-stable $\Acal_X^\dagger \otimes k(x)$-lattice.\\
\noindent {\rm (ii)} The subset $\{x\in X\mid \iota_x^\ast \Ncal\ \text{is \'etale}\}$ is open.\\
\noindent {\rm (iii)} The family $\Ncal$ is \'etale if and only if $\iota_x^\ast\Ncal$ is \'etale for all $x\in X$.
\end{cor}
\begin{proof}
We may assume (locally on $X$) that the family $\Ncal$ admits a model over $\Ubb_{\geq r}$. The claim now follows from Proposition $\ref{latticedecompletion}$, Theorem $\ref{loket}$ and Proposition $\ref{etaleonannulus}$.
\end{proof}
\begin{rem}\label{intgleichwa}
Note that it is not clear whether a family of $\phi$-modules over the Robba ring is \'etale if and only if it is \'etale in all rigid points.
This would be the case if one could establish a formalism of slope filtrations as in \cite{Kedlaya} in the fibers over all points of $X$, not just the rigid analytic ones.
\end{rem}
\begin{rem}
As in Example $\ref{noBerko}$ above we find that the \'etale locus is not a Berkovich space in general. Using the same notations as in example $\ref{noBerko}$, one can show that the family on $X$ obtained from the family of filtered isocrystals, by applying the construction in section $5$ and restricting to the boundary, is \'etale exactly over $X^{\rm wa}$.  
\end{rem}

\begin{prop}\label{etalegeometric}
Let $f:X\rightarrow Y$ be a morphism of adic spaces locally of finite type and let $\mathcal{N}_Y$ be a $\phi$-module over $\Bcal_Y^R$. Denote by $\mathcal{N}_X=f^\ast\Ncal_Y$ the pullback of $\mathcal{N}_Y$ to a $\phi$-module over $\Bcal_X^R$. Let
\begin{align*}
U&=\{x\in X\mid \mathcal{N}_X\ \text{is \'etale at}\ x\}\\
V&=\{y\in Y\mid \mathcal{N}_Y\ \text{is \'etale at}\ y\}.
\end{align*}
Then $f^{-1}(V)=U$.
\end{prop}
\begin{proof}
First it is obvious that $f^{-1}(V)\subset U$. Let $x\in U$ mapping to $y\in Y$. We have to show $y\in V$.
As above we may reduce to the case that $\Ncal_Y$ (and hence also $\Ncal_X$) admits a model over $\Ubb_{\geq r}$ and we may even work over $X\times \boldB_{[r,r_2]}$ and $Y\times\boldB_{[r,r_2]}$.\\
As $U$ is open, we may replace $x$ by a point in $U\cap f^{-1}(y)$ such that $E=k(x)$ is finite over $F=k(y)$.
Enlarging $E$ if necessary we may assume that the extension is finite and Galois with Galois group $H$.

Fix an $\Acal_E^{[r,r_2]}$-lattice in $\mathcal{N}\otimes k(x)$ that satisfies the condition $(\ref{etonannulus2})$.
Let $e_1,\dots,e_d$ be a basis of this lattice. We can consider the $\Acal_E^{[r,r_2]}$-lattice $N$ which is generated by the translates of the $e_i$ by the action of $H$.
Now $N$ is stable under the action of $H$ and we may take the invariants $N'=N^H$ which are a finitely generated $\Acal_F^{[r,r_2]}$-submodule that also satisfies condition $(\ref{etonannulus2})$ and which generates $\Ncal\otimes k(y)$ as $N$ contains a neighborhood of $0$. We claim that $N'$ is free.

As $N'$ is finitely generated, it is defined in a neighborhood \[U=\Spa(A,A^+)\subset Y\] of $y$ and generates $\Ncal|_U$. Especially $N'$ is freely generated by exactly $d$ elements in every rigid analytic point of $U\times\boldB_{[r,r_2]}$. Now we view $N'$ as a coherent module on the formal model $\Spf A^+\times \Xcal$ of $U\times\boldB_{[r,r_2]}$, where $\Xcal$ denotes the canonical formal model of $\boldB_{[r,r_2]}$ with special fiber $\Spec(A^+/A^{++})\times(\Abb^1_k\cup\Abb^1_k)$. Here $A^{++}\subset A^+$ denotes the ideal of topologically nilpotent elements. We consider the restriction $\bar N'$ of $N'$ to this special fiber and find that it is freely generated by $d$ elements in every closed point of the special fiber. It follows that the reduction of $N'\subset \Ncal_Y\otimes k(y)$ modulo the maximal ideal $\mfrak_{F^+}$ of $F^+$ is free of rank $d$ on $\Spec \kappa_F \times (\Abb^1_k\cup\Abb^1_k)$. By Nakayama's lemma we can lift $d$ generators of the special fiber to generators of $N'$. As $N'\otimes_{F^+}F$ is free of rank $d$, these generators can not satisfy any relations.
\end{proof}

\section{The period morphism}

In this section we deal with an integral model for the stack $\Dfrak_\nu$.
We recall from \cite{phimod} the definition of a stack of integral models $\widehat{\Ccal}_{K,\nu}$ and the definition of the period morphism
\[\Pi(\Xcal):\widehat{\Ccal}_{K,\nu}(\Xcal)\longrightarrow \Dfrak_\nu(\Xcal^{\rig})\]
for a $p$-adic formal scheme $\Xcal$.
First we recall from \cite{phimod} the definitions of the stacks $\widehat{\Ccal}_K$ and $\widehat{\Ccal}_{K,\nu}$ on the category ${\rm Nil}_p$ of $\Z_p$-schemes on which $p$ is locally nilpotent.

For a $\Z_p$-algebra $R$ such that $p^nR=0$ for some $n$ we set $\widehat{\Ccal}_K(R)$ to be the groupoid of pairs $(\Mfrak,\Phi)$, where $\Mfrak$ is an $R_W[[u]]$-module that is fpqc-locally on $\Spec\,R$ free over $R_W[[u]]$ and $\Phi:\Mfrak\rightarrow \Mfrak$ is a $\phi$-linear map such that
\[E(u)\Mfrak\subset \Phi(\phi^\ast \Mfrak)\subset \Mfrak.\] 
Here we write $R_W=R\otimes_{\Z_p}W$ and 
\[\phi:R_W[[u]]\longrightarrow R_W[[u]]\]
is the morphism that is the identity on $R$, the natural Frobenius on $W$ and that takes $u$ to $u^p$. 

Given a dominant coweight $\nu$ that satisfies the extra condition $(\ref{specialcochar})$, there is a closed substack $\widehat{\Ccal}_{K,\nu}\subset\widehat{\Ccal}_K$, where the $\Spf\,\Ocal_F$-valued points (for finite extensions $F$ of $\Q_p$ containing the reflex field of $\nu$) are the pairs $(\Mfrak,\Phi)$ such that the reduction modulo $E(u)$ of the filtration
\[E(u)\Phi(\phi^\ast\Mfrak)\subset E(u)\Mfrak\subset \Phi(\phi^\ast\Mfrak)\]
is of type $\nu$. See \cite[3.d]{phimod} for the precise definition of $\widehat{\Ccal}_{K,\nu}$.\\
These definitions imply that for a $p$-adic formal scheme $\Xcal$ topologically of finite type over $\Z_p$ we have a map
\[\widehat{\Ccal}_{K,\nu}(\Xcal)\longrightarrow \Cfrak_\nu(\Xcal^{\ad})\longrightarrow \Dfrak_\nu(\Xcal^{\ad})\]
which will be referred to as the \emph{period map}, see also \cite[5.b.4]{phimod}. Here the first arrow is given by analytification of a pair $(\Mfrak,\Phi)$.

The motivation to consider these stacks comes again from the theory of $p$-adic Galois representations.
By Kisin's classification of finite flat group schemes of $p$-power order (cf. \cite[1]{Kisin}), if $p>2$, the pairs $(\Mfrak,\Phi)\in\widehat{\Ccal}_K(R)$ correspond to finite flat group schemes over $\Spec\ \Ocal_K$ with an action of $R$, provided $R$ is finite and $\Mfrak$ is of projective dimension $1$ over $W[[u]]$. This motivates our definition of the stack $\widehat{\Ccal}_K$. If $p=2$ then there is a similar classification but it only applies to connected objects, where a $\phi$-module $(\Mfrak,\Phi)$ is called connected if $\Phi$ is topologically nilpotent (see \cite[Theorem 1.3.9]{Kisinp=2}).

Fixing a cocharacter $\nu$, a $\Spf\ \Ocal_F$-valued point of $\widehat{\Ccal}_{K,\nu}$ (for some $F$ containing the reflex field of $\nu$) defines a $p$-divisible group over $\Spec\,\Ocal_K$ with action of $\Ocal_F$. The idea of the definition of the stack $\widehat{\Ccal}_{K,\nu}$ is that the Hodge-filtration on the rational Dieudonn\'e module of this $p$-divisible group is of type $\nu$.

First we study what the period morphism does on $\Spf\,\Ocal_F$ valued points. We write $\Dfrak_K$ for the disjoint union of the $\Dfrak_\nu$ for all (miniscule) cocharacters $\nu$ as in $(\ref{specialcochar})$. The definitions imply that $\widehat{\Ccal}_{K,\nu}$ maps to $\Dfrak_\nu\subset\Dfrak_K$ under the period morphism.\\
We make the following observations that are already contained in Kisin's paper \cite{crysrep}.
\begin{prop}\label{periodim}
Let $F$ be a finite extension of $\Q_p$ with ring of integers $\Ocal_F$. A filtered $\Phi$-module $(D,\Phi,\Fcal^\bullet)\in\Dfrak_K(F)$ is weakly admissible if and only if it is in the image of the period morphism, i.e. there is an object $x=(\Mfrak,\Phi)\in \widehat{\Ccal}_K(\Ocal_F)$ such that 
\[(D,\Phi,\Fcal^\bullet)=\Pi(\Spf\, \Ocal_F)(x).\] 
\end{prop}
\begin{proof}
This is a consequence of \cite[Corollary 10.4.8, 10.4.9]{survey}.
\end{proof}
Further we remark the following connection between the image of an $\Ocal_F$-valued point $(\Mfrak,\Phi)$ of $\widehat{\Ccal}_K$ and the Galois representation on the Tate module of the $p$-divisible group associated to $(\Mfrak,\Phi)$.
\begin{prop}
Let $F$ be a finite extension of $\Q_p$ with ring of integers $\Ocal_F$ and $x=(\Mfrak,\Phi)\in\widehat{\Ccal}_K(\Ocal_F)$. Write $V_x=T_p(\Gcal_x)\otimes_{\Z_p}\Q_p$ for the rational Tate-module of the associated $p$-divisible group and $(D,\Phi,\Fcal^\bullet)\in\Dfrak_K(F)$ for the image of $x$ under the period morphism.
Then there is a natural isomorphism of filtered $\Phi$-modules 
\[D_{\rm cris}(V_x(-1))\cong (D,\Phi,\Fcal).\]
\end{prop}
\begin{proof}
As $(D,\Phi,\Fcal^\bullet)$ is weakly admissible, there is a crystalline representation $V=V_{\rm cris}(D)$ associated to it, by \cite[Theorem A]{ColmezFont}. By \cite[Corollary 2.1.14]{crysrep} the restriction functor to $G_{K_\infty}={\rm Gal}(\bar K/K_\infty)$ is fully faithful on the category of crystalline $G_K$-representation. Hence it is enough to show that 
\[V_x(-1)|_{G_{K_\infty}}\cong V|_{G_{K_\infty}}.\]
But with the notations of $(\ref{GKinftyrepn})$ we have isomorphisms 
\[V_x(-1)|_{G_{K_\infty}}\longrightarrow V_{{\bf B}}(\Mfrak\otimes_{{\bf A}^{[0,1)}}{{\bf B}})\longrightarrow  V|_{G_{K_\infty}},\]
where the first isomorphism follows from \cite[Proposition 1.1.13]{Kisin} and the second follows from \cite[Proposition 11.3.3]{survey}.
\end{proof}
\begin{rem}
By Kisin's classification, the objects $(\Mfrak,\Phi)\in\widehat{\Ccal}_K(\Spf\,\Ocal_F)$ correspond to $p$-divisible groups over $\Spec\,\Ocal_K$, see \cite[Corollary 2.2.22]{Kisin}. In the above Proposition we use implicitly that every crystalline representation with Hodge-Tate weights in $\{0,1\}$ arises as the Tate-module of some $p$-divisible group over $\Spec\,\Ocal_K$. This was shown by Breuil for $p\neq 2$ in \cite[Theorem 5.3.2]{pdivBreuil} and by Kisin in \cite[Theorem 0.3]{crysrep} for arbitrary $p$.
\end{rem}
Denote by $E$ the reflex field of the cocharacter $\nu$. We now define a stack $\Dfrak_\nu^{\rm int}$ on the category ${\Ad}^{\lft}_E$ which will be the image of the period morphism in a sense specified below.
We will write 
\[\Acal_X^{[0,1)}=\pr_{X,\ast}\Ocal_{X\times\Ubb}^+\]
for the sheafified version of the ring ${\bf A}^{[0,1)}$ on an adic space $X$.
\begin{defn}\label{height1modules}
A \emph{$\phi$-module of height $1$} over $\Acal_X^{[0,1)}$ is an $\Acal_X^{[0,1)}$-module $\Mfrak$ that is locally on $X$ free of finite rank with an injection $\Phi:\phi^\ast\Mfrak\hookrightarrow \Mfrak$ such that $E(u)\coker \Phi=0$.
\end{defn}
\begin{defn}
Let $\Dfrak_\nu^{\rm int}$ denote the groupoid that assigns to an adic space $X\in\Ad_E^{\rm lft}$ the groupoid of $(D,\Phi,\Fcal^\bullet)\in\Dfrak_\nu(X)$
such that locally on $X$ there exists a $\phi$-module $\Mfrak$ of height $1$ over $\Acal_X^{[0,1)}$ such that 
\[\Mfrak\otimes_{\Acal_X^{[0,1)}}\Bcal_X^{[0,1)}\cong \underline{\Mcal}(D,\Phi,\Fcal^\bullet),\]
where $\underline{\Mcal}(D)$ is the vector bundle on $X\times\Ubb$ defined in $(\ref{mapM})$.
\end{defn}
\begin{theo}\label{maintheoint}
\noindent {\rm (i)} Let $X\in\Ad_E^{\rm lft}$ and $(D,\Phi,\Fcal^\bullet)\in\Dfrak_\nu(X)$ be a family of filtered $\phi$-modules on $X$. Then $(D,\Phi,\Fcal^\bullet)\in\Dfrak_\nu^{\rm int}(X)$ if and only if the family 
\[\underline{\Mcal}(D)\otimes_{\Bcal_X^{[0,1)}}\Bcal_X^R\] 
of $\phi$-modules over the Robba ring is \'etale.\\
\noindent {\rm (ii)} The groupoid $\Dfrak_\nu^{\rm int}$ is an open substack of $\Dfrak_\nu^{\rm wa}$ and $\Dfrak_\nu^{\rm int}(F)=\Dfrak_\nu^{\rm wa}(F)$ for any finite extension $F$ of $E$.  
\end{theo}
\begin{proof}
\noindent (i) The one implication is obvious. Assume that the family is \'etale.
As being \'etale may be checked at points and is compatible with base change we can reduce to the universal case and assume that $X$ is reduced.
Locally on $X$ there is $0<r<1$ and an $\Acal_X^{[r,1)}$-lattice $N$ in 
\[\underline{M}(D)\otimes_{\Bcal_X^{[0,1)}}\Bcal_X^{[r,1)}=\underline{\Mcal}(D)|_{X\times \Ubb_{\geq r}} \]
such that 
\[\Phi(\phi^\ast N)=N|_{X\times \Ubb_{\geq r_1}},\]
where we again write $r_i=r^{1/p^i}$.
We also write $N$ for the restriction of $N$ to the annulus $X\times\boldB_{[r,r_2]}$.

By Proposition $\ref{intmodel}$ below we can (locally on $X$) construct a free $\Acal_X^{[0,r_2]}$-module 
\[M_r\subset \Mcal(D)\otimes_{\Bcal_X^{[0,1)}}\Bcal_X^{[0,r_2]} \]
such that 
\[E(u)M_r\subset \Phi(\phi^\ast M_r|_{X\times \boldB_{[r,r_1]}})\subset M_r.\]
By gluing $M_r$ with $\phi^\ast M_r|_{X\times \boldB_{[r_2,r_3]}}$ using the isomorphism $\Phi$, we can extend $M_r$ to $\Acal_X^{[0,r_3]}$. Repeating this procedure yields the desired $\phi$-module of height $1$ over $\Acal_X^{[0,1)}$. The freeness of this module follows by the same argument as in Proposition $\ref{etaleonannulus}$\\
\noindent (ii) By Corollary $\ref{etaleopen}$ the condition of being \'etale is equivalent to being \'etale at all points, which is an open condition. The fact that this notion is fpqc-local follows from Proposition $\ref{etalegeometric}$. The final fact is a consequence of \cite[Theorem 1.3.8]{crysrep}.
\end{proof}
Recall that we write  
\[{\bf A}_A^{[r,s]}=\Gamma(\Spa(A,A^\circ),\Acal_{\Spa(A,A^\circ)}^{[r,s]})=A^\circ\widehat{\otimes}_{\Z_p}{\bf A}^{[r,s]}=A^\circ\langle T/s, r/T \rangle,\]
for a complete $\Q_p$-algebra $A$ topologically of finite type.
\begin{prop}\label{intmodel}
Let $X=\Spa\,(A,A^+)$ be a reduced affinoid adic space of finite type over $E$. 
Let $r>|\pi|$ with $r\in p^\Q$ and set $r_i=r^{1/p^i}$. Let $\Mcal_r$ be a free vector bundle on $X\times\boldB_{r_2}$ such that $E(u)\Mcal_r\subset \Phi(\phi^\ast\Mcal_r|_{X\times \boldB_{r_2}})\subset \Mcal_r$. Assume that there exists a free ${\bf A}_A^{[r,r_2]}=A^+\langle T/r_2,r/T\rangle$ submodule 
\[N_r\subset \Ncal_r:=\Mcal_r\otimes_{{\bf A}_A^{[0,r_2]}}{\bf A}_A^{[r,r_2]}\]
of rank $d$, containing a basis of $\Ncal_r$ such that 
\[\Phi(\phi^\ast(N_r\otimes_{{\bf A}_A^{[r,r_2]}} {\bf A}_A^{[r,r_1]}))=N_r\otimes_{{\bf A}_A^{[r,r_2]}} {\bf A}_A^{[r_1,r_2]}. \]
Then locally on $X$ there exists a free ${\bf A}_A^{[0,r_2]}$-submodule $M_r\subset \Mcal_r$ of rank $d$, containing a basis of $\Mcal_r$ such that
\begin{equation}\label{phicond}
\Phi:\phi^\ast(M_r\otimes_{{\bf A}_A^{[0,r_2]}}{\bf A}_A^{[0,r_1]})\longrightarrow M_r
\end{equation}
is injective with cokernel killed by $E(u)$.
\end{prop}
\begin{proof}
The idea of the proposition is to give a relative version of the proof of \cite[Lemma 1.3.13]{crysrep}.\\
We also write $\Mcal_r$ for the global sections of the vector bundle. Write $M'_r=\Mcal_r\cap N_r\subset \Ncal_r$. This is an $A^+\langle T/r_2\rangle$-module. Further we set
\[M_r=(M'_r\otimes_{{\bf A}_A^{[0,r_2]}}{\bf A}_A^{[r,r_2]})\cap M'_r[\tfrac{1}{p}]\subset \Ncal_r.\]
Then $M_r$ is a finitely generated $A^+\langle T/r_2\rangle$-module as the ring is noetherian, and we claim that it is locally on $X$ free and satisfies $(\ref{phicond})$.
By the construction of $M_r$ the map in $(\ref{phicond})$ exists and for all quotients $A^+\rightarrow \Ocal_F$ onto a finite flat $\Z_p$-algebra, the $\Ocal_F\langle T/r_2\rangle$-module $M_r\otimes_{A^+}\Ocal_F$ is free and satisfies the condition in $(\ref{phicond})$ by the argument in the proof of \cite[Lemma 1.3.13]{crysrep}.
It now follows that the reduction $\bar M_r$ of $M_r$ modulo $\varpi_E$ defines a vector bundle on $\Spec(A^+/\varpi_EA^+)^{\rm red}\times \Abb^1$. Hence, after localization on $X$, we may assume that $\bar M_r$ it is free. As $M_r$ is $\Ocal_E$-flat (by construction it has no $\varpi_E$-power torsion) the claim follows by Nakayama's Lemma after lifting a basis of $\bar M_r$ to $M_r$.
\end{proof}
We now want to show that the stack $\Dfrak_\nu^{\rm int}$ is the image of the period map defined by Pappas and Rapoport in the sense of (a slightly weakened version) of \cite[Conjecture 5.3]{phimod}.

As the stacks $\Dfrak_\nu$, $\Dfrak_\nu^{\rm wa}$ and $\Dfrak_\nu^{\rm int}$ are Artin stacks, it makes sense to extend them to the category of all adic spaces (not necessarily of finite type) over the reflex field of $\nu$.

Let $F$ be a complete topological field over $\Q_p$ and $v_F:F^\times\rightarrow \Gamma_F$ a continuous valuation. As usual we write $F^+=\{a\in F\mid v_F(a)\leq 1\}$ for the ring of integral elements.

We say that $F$ is \emph{topologically finitely generated} if there exist finitely many elements $f_1,\dots,f_m\in F$ such that $\Frac(\Q_p[f_1,\dots,f_m])$ is dense in $F$.

We say that $F$ is \emph{of} $p$\emph{-adic type} if for any $f_1,\dots,f_m\in F$ the topological closure $A$ of the $\Q_p$-subalgebra $\Q_p[f_1,\dots,f_m]$ in $F$ is Tate, i.e. a quotient of some $\Q_p\langle T_1,\dots, T_{m'}\rangle$, and the intersection of $F^+$ with $A$ is precisely the ring of power bounded elements $A^\circ\subset A$.

If $F$ is topologically finitely generated and of $p$-adic type, then there is an adic space $X$ of finite type over $\Q_p$ such that $F$ is the completion of the residue field at a point of $X$: If $f_1,\dots,f_m$ are topological generators and if $A$ denotes the closure of $\Q_p[f_1,\dots,f_m]$, then the restriction of $v_F$ to $A$ defines a point in $\Spa(A,A^\circ)$ such that the completion of its residue field is $F$.  
Conversely, if $X$ is an adic space, then the completion of the residue field at a point $x\in X$ is topologically finitely generated and of $p$-adic type.

\begin{theo}
\noindent {\rm (i)} Let $F$ be a topologically finitely generated field of $p$-adic type and $(D,\Phi,\Fcal^\bullet)\in \Dfrak_\nu(F)$. Then there exists $(\Mfrak,\Phi)\in\widehat{\Ccal}_{K,\nu}(\Spf\,F^+)$ such that $\Pi(\Spf\,F^+)(\Mfrak,\Phi)=(D,\Phi,\Fcal^\bullet)$ if and only if 
\[\underline{\Mcal}(D)\otimes_{\Bcal_F^{[0,1)}}\Bcal_F^R\]
is \'etale if and only if $\Spa(F,F^+)\rightarrow \Dfrak_\nu$ factors over $\Dfrak_\nu^{\rm int}$.\\
\noindent {\rm (ii)} Let $X\in\Ad_{\Q_p}^{\rm lft}$ and $f:X\rightarrow \Dfrak_\nu$ a morphism defined by $(D,\Phi,\Fcal^\bullet)$. Then $f$ factors over $\Dfrak_\nu^{\rm int}$ if and only if there exists a covering $U_i$ of $X$ and formal models $\mathcal{U}_i$ of $U_i$ together with $(\Mfrak_i,\Phi_i)\in\widehat{\Ccal}_{K,\nu}(\mathcal{U}_i)$
such that
\[\Pi(\mathcal{U}_i)(\Mfrak_i,\Phi_i)=(D,\Phi,\Fcal^\bullet)|_{U_i}.\] 
\end{theo}
\begin{proof}
\noindent (i) If $(\Mfrak,\Phi)$ exists, then it is obvious that 
\[\underline{\Mcal}(D)\otimes_{\Bcal_F^{[0,1)}}\Bcal_F^R\]
is \'etale. Conversely, assume this is the case. By the above remark we may embed the field $F$ into a space $X$ of finite type and (after localizing) assume that $(D,\Phi,\Fcal^\bullet)$ extends to a family on $X$. Then \'etaleness is equivalent to the fact that $X\rightarrow \Dfrak_\nu$ factors locally around $x$ over $\Dfrak_\nu^{\rm int}$ by Theorem $\ref{maintheoint}$ and this also implies the existence of a lattice of $\phi$-modules of height $1$ over $\Acal_X^{[0,1)}$ locally around $x$. The fiber of this lattice at $x$ then defines the point $(\Mfrak,\Phi)\in\widehat{\Ccal}_{K,\nu}(\Spf\,F^+)$. \\
\noindent (ii) If $X$ is reduced this is just a restatement of Theorem $\ref{maintheoint}$, using the fact that a free $\Acal_X^{[0,1)}$-module on an reduced, affinoid adic space $X=\Spa(A,A^+)$ is the same as a free $A^+\widehat{\otimes}_{\Z_p}W[[u]]$-module and $\Spf\,A^+$ is a formal model for $X$. The general case follows by locally choosing a formal model $f:\mathcal{Y}\rightarrow \mathcal{X}$ for the morphism to the universal family and pulling back the family of $\phi$-modules over $\Ocal_{\mathcal{X}}\widehat{\otimes}_{\Z_p}W[[u]]$ along $f$.  
\end{proof}

\section{Families of crystalline representations}

In the previous section we have constructed an open substack of the stack of filtered $\phi$-modules on which there exists an integral model (in the sense of Kisin's paper \cite{crysrep}) for the (weakly admissible) filtered $\phi$-module, if the type of the filtration is given by a miniscule coweight. Now we want to construct an open substack of the weakly admissible locus on which there exists a family of crystalline representations giving rise to the restriction of the universal filtered $\phi$-module.

\subsection{Some sheaves of period rings}
First we have to define some sheaves of topological rings which are relative versions of the period rings recalled in section 2. Recall that an adic space $X$ comes along with an open and integrally closed subsheaf $\Ocal_X^+$ of the structure sheaf $\Ocal_X$. Unfortunately this subsheaf is only well behaved if $X$ is reduced: if $f\in\Ocal_X$ is nilpotent, then $p^{-n}f\in\Ocal_X^+$ for all $n$. Hence we will consider the integral versions of the period sheaves only for reduced spaces.
It will turn out that this is enough, as we have a reduced universal case.

The ring $R$ introduced in section 2 is equipped with a valuation ${\rm val}_R$ given by
\[{\rm val}_R((x_0,x_1,x_2,\dots))=v_p(x_0),\]
where $x_i\in\Ocal_{\mathbb{C}_p}$, and $x_i^p=x_{i-1}$ for all $i$. The topology defined by the valuation coincides with the canonical topology of $R$, and $R$ is complete
with respect to this topology.

Let $A^+$ be a reduced, $p$-adically complete, flat $\Z_p$-algebra topologically of finite type over $\Z_p$. We define
\[A^+\widehat{\otimes}W(R)=\lim\limits_{\substack{\longleftarrow \\ i\geq 0}}A^+\widehat{\otimes}_{\Z_p}W_i(R),\]
where $A^+\widehat{\otimes}_{\Z_p}W_i(R)$ is the completion of $A^+\otimes_{\Z_p}W_i(R)$ with respect to the discrete topology on $A^+/p^iA^+$
and the natural topology on the truncated Witt vectors $W_i(R)=W(R)/p^iW(R)$.
\begin{lem}\label{basicproperty}
\noindent {\rm (i)} Let $A^+$ and $B^+$ be two $\Z_p$-algebras as above and assume $B^+=A^+/I$ for some (closed) ideal $I\subset A^+$. Then
\[B^+\widehat{\otimes}W(R)=A^+\widehat{\otimes}W(R)/I(A^+\widehat{\otimes}W(R)).\]
\noindent {\rm (ii)} If $B^+$ is a finite $A^+$-algebra, then
\[B^+\widehat{\otimes}W(R)=(A^+\widehat{\otimes}W(R))\otimes_{A^+}B^+.\]
\noindent {\rm (iii)} Let $X$ be an reduced, affinoid adic space of finite type over $\Q_p$ and $X=\bigcup U_i$ be a covering by open, affinoid subspaces.
Write $A^+=\Gamma(X,\Ocal_X^+)$. Further set $A_i^+=\Gamma(U_i,\Ocal_X^+)$ and $A_{ij}^+=\Gamma(U_i\cap U_j,\Ocal_X^+)$. Then
\[0\longrightarrow A^+\widehat{\otimes}W(R)\longrightarrow \prod_i A_i^+\widehat{\otimes}W(R)\substack{\longrightarrow \\ \longrightarrow }\prod_{i,j}A_{ij}^+\widehat{\otimes}W(R)\]
is exact.
\end{lem}
\begin{proof}
\noindent (i), (ii) These are direct consequences of the definition.\\
\noindent (iii) As $\Ocal_X^+$ is a sheaf, we find that 
\[0\longrightarrow A^+\longrightarrow \prod_i A_i^+\substack{\longrightarrow \\ \longrightarrow }\prod_{i,j}A_{ij}^+\]
is exact. This sequence is still exact if we reduce modulo $p^i$ and tensorize with $W_i(R)$. Further it stays exact if we complete with respect to the topology on $W_i(R)$. The claim follows from this.
\end{proof}
Using this Lemma, we can define a sheaf $\Ocal_X^+\widehat{\otimes}W(R)$ on a reduced, affinoid adic space $X$ of finite type over $\Q_p$ such that
\[\Gamma(U,\Ocal_X^+\widehat{\otimes}W(R))=\Gamma(U,\Ocal_X^+)\widehat{\otimes}W(R)\]
for all affinoid open subsets $U\subset X$: For an arbitrary open subset $V\subset X$ we define 
\[\Gamma(V,\Ocal_X^+\widehat{\otimes}W(R))=\lim\limits_{\substack{ \longleftarrow \\ V\supset U}}\Gamma(U,\Ocal_X^+)\widehat{\otimes}W(R),\]
where the limit is taken over all open, affinoid subspaces $U\subset V$.
Further we can extend this construction to all reduced adic spaces locally of finite type. Similarly we can define the sheaf $\Ocal_X^+\widehat{\otimes} W(\Frac R)$.

If $X\in\Ad^{\lft}_{\Q_p}$, then we can consider the sheaf $\Ocal_X\widehat{\otimes}W(\Frac R)[1/p]$, which is obtained from $\Ocal_X^+\widehat{\otimes}W(R)$ by inverting $p$. This sheaf also makes sense if $X$ is not reduced, where it is defined to be the restriction of $\Ocal_Y\widehat{\otimes}W(\Frac R)[1/p]$ to $X$ for some closed embedding of $X$ into a reduced space $Y$.

Given a section $f$ of $\Ocal_X\widehat{\otimes}W(\Frac R)[1/p]$, we say that $f$ vanishes at a rigid point $x\in X$ if 
\[f(x)\in (\Ocal_X\widehat{\otimes}W(\Frac R)[1/p])\otimes k(x)\]
is zero. We use the same terminology for the other sheaves of period rings that we are going to construct.
\begin{lem}\label{localizprinc1}
Let $A^+=\Z_p\langle T_1,\dots,T_n\rangle$ and $X=\Spa(A^+[1/p],A^+)$. \\
\noindent {\rm (i)}  The canonical map
\[A^+\widehat{\otimes}W(R)\longrightarrow \prod k(x)^+\otimes_{\Z_p}W(R),\]
where the product runs over all rigid points $x\in X$, is an injection.\\
\noindent {\rm (ii)} Let 
\[\alpha: (A^+\widehat{\otimes}W(R))^{m_1}\longrightarrow (A^+\widehat{\otimes}W(R))^{m_2} \]
be an $A^+\widehat{\otimes}W(R)$-linear map such that $\coker\, \alpha$ vanishes at all rigid points of $X$. Then $\coker\, \alpha=0$.
\end{lem}
\begin{proof}
\noindent (i) Let $f\in A^+\widehat{\otimes}W(R)$. After dividing by some power of $p$, we may assume that 
\[f\mod p=\bar f\in\Fbb_p[T_1,\dots,T_n]\widehat{\otimes}_{\Fbb_p}R\]
is non zero. We need to check that the image of $\bar f$ in $\prod \kappa(x)\otimes_{\Fbb_p} R$ is non zero, where the product runs over all closed points $x$ of $\Spec\,\Fbb_p[T_1,\dots,T_n]$.\\
Fix an $\Fbb_p$-basis $(b_i)_{i\in I}$ of $R$, then 
\[\bar f=\sum_{i\in I}\bar f_i\otimes b_i,\]
with $\bar f_i\in\Fbb_p[T_1,\dots,T_n]$ and the sets $\{i\in I\mid \bar f_i\neq 0\ \text{and}\ {\rm val}_R(b_i)\leq r\}$ are finite for all $r\in \R$.
As $\bar f\neq 0$ there exists an index $i\in I$ such that $\bar f_i\neq 0$. Hence there exists a closed point $x\in\Spec\,\Fbb_p[T_1,\dots,T_n]$ such that $\bar f_i(x)\neq 0$. It follows that the image of $\bar f$ in $\prod \kappa(x)\otimes_{\Fbb_p} R$ does not vanish.\\
\noindent (ii) Again we can reduce the claim modulo $p$ and obtain a map
\[\bar\alpha: (\Fbb_p[T_1,\dots,T_n]\widehat{\otimes}R)^{m_1}\longrightarrow (\Fbb_p[T_1,\dots,T_n]\widehat{\otimes}R)^{m_2}\]
such that $\coker\, \bar\alpha$ vanishes at all closed points $x\in\Spec\,\Fbb_p[T_1,\dots,T_n]$. Now $\coker\, \alpha$ is a finitely generated module over $\Fbb_p[T_1,\dots,T_n]\widehat{\otimes}_{\Fbb_p}R$. We denote by $I$ the ideal of $\Fbb_p[T_1,\dots,T_n]\widehat{\otimes}_{\Fbb_p}R$ which is generated by all elements in $R$ with positive valuation. Then $I$ is contained in every maximal ideal and by Nakayama's lemma $\coker\bar\alpha$ vanishes if its reduction modulo $I$ vanishes. But $(\Fbb_p[T_1,\dots,T_n]\widehat{\otimes}_{\Fbb_p}R)/I=\Fbb_p[T_1,\dots,T_n]\otimes_{\Fbb_p}\bar\Fbb_p$ and every finitely generated module over this ring clearly vanishes if it vanishes modulo all maximal ideals of $\Fbb_p[T_1,\dots,T_n]$.  
\end{proof}
\begin{rem}\label{torsionlocaliz}
The same proofs as above also apply to $p$-torsion modules, i.e. if $M$ is a finitely presented $(A^+\widehat{\otimes}W(R))/p^m$-module which vanishes at all the rigid analytic points of $X$, i.e. $M\otimes_{\Z_p}k(x)^+=0$ for all rigid analytic points  $x\in X$, then $M=0$. Further, if $M\subset N$ is an inclusion of finitely presented $(A^+\widehat{\otimes}W(R))/p^m$-modules and $f\in M$ such that 
\[f(x)\in N\otimes_{\Z_p}k(x)^+\subset M\otimes_{\Z_p}k(x)^+\]
for all rigid points $x$, then $f\in N$.
\end{rem}
\begin{cor}\label{localizprinc2}
Let $X$ be an affinoid adic space of finite type over $\Q_p$. By Noether normalization there exists a finite morphism $f:X\rightarrow Y$, where $Y=\Spa(\Q_p\langle T_1,\dots, T_n\rangle,\Z_p\langle T_1,\dots,T_n \rangle)$ is a polydisc. Then the canonical map
\[\Gamma(X,\Ocal_X\widehat{\otimes}W(R)[\tfrac{1}{p}])\longrightarrow \prod\Gamma(f^{-1}(y),\Ocal_X)\otimes_{\Q_p}W(R)[\tfrac{1}{p}]\]
is an injection, where the product runs over all rigid analytic points of $Y$. Further a finitely presented $\Ocal_X\widehat{\otimes}W(\Frac R)[1/p]$-module vanishes if it vanishes at all rigid points of $Y$.
\end{cor}
\begin{proof}
As $\Ocal_X$ is a finite $\Ocal_Y$-module, this is consequence of lemma $\ref{localizprinc1}$.
\end{proof}

Let $X\in\Ad^{\lft}_{\Q_p}$ be a reduced space. We define the sheaf $\widetilde{\Acal}_X$ by taking the closure of $\Ocal_X^+\otimes_{\Z_p}\widetilde{\bf A}$ in $\Ocal_X^+\widehat{\otimes}W(\Frac R)$. By this we mean that for an affinoid open $U\subset X$ we define $\Gamma(U,\widetilde{\Acal}_X)$ to be the topological closure of 
\[\Gamma(U,\Ocal_X^+)\otimes_{\Z_p}\widetilde{\bf A}\subset \Gamma(U,\Ocal_X^+\widehat{\otimes}W(\Frac R)).\]
Similarly we define the sheaves $\widetilde{\Acal}_X^{[0,1)}$ and $\widetilde{\Bcal}_X$ by taking the closures of
\begin{align*}
\Gamma(U,\Ocal_X^+)\otimes_{\Z_p}\widetilde{\bf A}^{[0,1)}&\subset \Gamma(U,\Ocal_X^+\widehat{\otimes}W(R)),\ \ \text{respectively}\\
\Gamma(U,\Ocal_X)\otimes_{\Q_p}\widetilde{\bf B}&\subset \Gamma(U,\Ocal_X\widehat{\otimes}W(\Frac R)[\tfrac{1}{p}]).
\end{align*}
Further we define the sheaf $\Acal_X$ as follows: If $X=\Spa(A,A^+)$ is an affinoid adic space, then
\[\Gamma(X,\Acal_X):=\Gamma(X,\Acal_X^{[0,1)})[\tfrac{1}{u}]^\wedge=(A^+\otimes_{\Z_p}W)((u))^\wedge,\]
where $(-)^\wedge$ means $p$-adic completion. As in the construction of $\Ocal_X^+\widehat{\otimes}W(R)$ we can extend this definition to a sheaf $\Acal_X$ on $X$.\\
Finally we define $\Bcal_X=\Acal_X[1/p]$.
Again we can consider $\Bcal_X$ and $\widetilde{\Bcal}_X$ on any (not necessarily reduced) adic space $X$.
By construction all these sheaves are endowed with an $\Ocal_X$ (resp. $\Ocal_X^+$) linear continuous Frobenius $\phi$. Further we have an action of $G_{K_\infty}$ on $\widetilde{\Acal}_X$ and $\widetilde{\Bcal}_X$.
\begin{rem}
The obvious analogue of Lemma $\ref{localizprinc1}$ and Corollary $\ref{localizprinc2}$ also holds true for the subsheaves of $\Ocal_X\widehat{\otimes}W(\Frac R)[1/p]$ introduced above.
\end{rem}
\begin{lem}\label{enoughrigpts}
Let $A^+=\Z_p\langle T_1,\dots,T_n\rangle$ and $X=\Spa(A^+[1/p],A^+)$. Let $S\in\{\Z_p,W,{\bf A}^{[0,1)},{\bf A},\widetilde{\bf A}^{[0,1)},\widetilde{\bf A}\}$ be a closed subring of $W(\Frac R)$ and let $f\in A^+\widehat{\otimes}W(\Frac R)$ such that $f(x)\in k(x)^+\otimes_{\Z_p}S$ for all rigid points $x\in X$. Then $f$ is in the closure of $A^+\otimes_{\Z_p}S\subset A^+\widehat{\otimes}W(\Frac R)$.
\end{lem}
\begin{proof}
By $p$-adic approximation we may reduce this claim modulo $p$. Write $\bar S=S/pS$. Fix an $\Fbb_p$-basis $(b_i)_{i\in I}$ of $\Frac R$ such that $(b_i)_{i_\in J}$ is a basis of $\bar S$ for a certain subset $J\subset I$. We can expand the reduction $\bar f$ of $f$ modulo $p$ in a series
\[\bar f=\sum_{i\in I}\bar f_i\otimes b_i.\]
Our assumption implies that $\bar f_i(x)=0$ for all $i\notin J$ and all closed points $x\in\Spec\,\Fbb_p[T_1,\dots,T_n]$. It follows that $\bar f_i=0$ for $i\notin J$.
\end{proof}
\begin{rem}
With the same notation as in the lemma, write $A=A^+[1/p]$ and let $B$ be a finite $A$-algebra. The the above implies that 
\[f\in B\widehat{\otimes}W(\Frac R)[1/p]\]
lies in the closure of $B\otimes_{\Q_p}{\bf B}$ inside $B\widehat{\otimes}W(\Frac R)[1/p]$, if 
\[f(x)\in (B\otimes_A k(x))\otimes_{\Q_p}{\bf B}\] 
for all rigid points $x\in \Spa(A,A^+)$. The same statement is also true, if we replace ${\bf B}$ by any of the other rings of the lemma.
\end{rem}
\begin{cor}\label{invariantsheaves}
Let $X$ be an adic space locally of finite type over $\Q_p$ and $x\in X$ be a rigid point. Then 
\begin{align*}
\widetilde{\Bcal}_X\otimes k(x)&=\widetilde{\bf B}\otimes_{\Q_p}k(x), & \Bcal_X\otimes k(x)&={\bf B}\otimes_{\Q_p}k(x).
\end{align*} Further 
\begin{align*}
\widetilde{\Bcal}_X^{\phi=\id}&= \Ocal_X, & \widetilde{\Bcal}_X^{G_{K_\infty}}&=\Bcal_X.
\end{align*}
If $X$ is reduced, a similar statement is true for the sheaf of integral rings.
\end{cor}
\begin{proof}
The first point is a direct consequence of Lemma $\ref{basicproperty}$. The second point is local on $X$ so we may assume that $X=\Spa(A,A^+)$ is affinoid.
By the Noether normalization theorem $X$ is finite over a polydisc $Y=\Spa(B,B^+)$. As $A$ is finite over $B$ we find that 
\[\Gamma(X,\widetilde{\Bcal}_X)=\Gamma(Y,\widetilde{\Bcal}_Y)\otimes_BA,\]
As in Corollary $\ref{localizprinc2}$ we have
\begin{equation}\label{inclusion}
\begin{aligned}
\Gamma(X,\widetilde{\Bcal}_X)\hookrightarrow &\prod\Gamma(f^{-1}(y),\Ocal_X)\otimes_{\Q_p}\widetilde{\bf B}\\
 =&\prod (A\otimes_Bk(y))\otimes_{\Q_p}\widetilde{\bf B},
\end{aligned}
\end{equation}
where the product runs over all rigid points $y\in Y$, and this inclusion is compatible with the action of $\phi$ and $G_{K_\infty}$. But, as $A\otimes_B k(y)$ is a finite dimensional $k(y)$-vector space we find that  
\[\big((A\otimes_Bk(y))\otimes_{\Q_p}\widetilde{\bf B}\big)^{G_{K_\infty}}=(A\otimes_{B}k(y))\otimes_{\Q_p}{\bf B}.\]
By Lemma $\ref{enoughrigpts}$ the image of $\Gamma(X,\Bcal_X)\subset \Gamma(X,\widetilde{\Bcal}_X)$ under $(\ref{inclusion})$ is identified with 
\[\Gamma(X,\widetilde{\Bcal}_X)\cap\left(\prod((A\otimes_Bk(y))\otimes_{\Q_p}\widetilde{\bf B}\right)^{G_{K_\infty}}=\Gamma(X,\widetilde{\Bcal}_X)^{G_{K_\infty}}\]
and the claim follows. The same argument applies to the other equality as well.
\end{proof}
\begin{rem}
If $x\in X$ is any point, and $\mathscr{R}$ is any of the sheaves considered above, is makes sense to define the sheaf $\iota_x^\ast\mathscr{R}$ on $\Spa(k(x),k(x)^+)$: We define it to be the completion of $\mathscr{R}\otimes k(x)$ with respect to its canonical topology. Sometimes we will also write $\mathscr{R}_{\widehat{k(x)}}$ for $\iota_x^\ast\mathscr{R}$.
\end{rem}

We now need to construct a relative version of Fontaine's period ring $B_{\rm cris}$. Again we first consider the case of a reduced space $X$ locally of finite type over $\Q_p$.
The map $\theta: W(R)\rightarrow \Ocal_{\mathbb{C}_p}$ extends to a map
\[\theta_X:\Ocal_X^+\widehat{\otimes}W(R)\longrightarrow \Ocal_X^+\widehat{\otimes}_{\Z_p}\Ocal_{\mathbb{C}_p},\]
where the completed tensor product on the right denotes the $p$-adic completion. We then define $\Ocal_X^+\widehat{\otimes}A_{\rm cris}$ to be the $p$-adic completion of the divided power envelope of $\Ocal_X^+\widehat{\otimes}W(R)$ with respect to $\ker \theta_X$. Finally we define 
\begin{align*}
\Ocal_X\widehat{\otimes}B_{\rm cris}^+&=\Ocal_X^+\widehat{\otimes}A_{\rm cris}[1/p] &  \Ocal_X\widehat{\otimes}B_{\rm cris}&=\Ocal_X\widehat{\otimes}B_{\rm cris}[1/t].
\end{align*}
These sheaves again make sense for any adic space $X\in\Ad_{\Q_p}^{\lft}$. Further there is a continuous $\Ocal_X$-linear Frobenius $\phi$ on $\Ocal_X\widehat{\otimes}B_{\rm cris}$ and a filtration on $(\Ocal_X\widehat{\otimes}B_{\rm cris})\otimes_{K_0}K$.
\begin{rem}
The notation $A^+\widehat{\otimes}A_{\rm cris}$ might be a little misleading, as this is not the completion of the ordinary tensor product over $\Z_p$ for the $p$-adic topology. The reason that it is not enough to consider this completion is that it does not contain $A^+[[u]]$, where $u=[\underline{\pi}]$. 
\end{rem}
\begin{lem}\label{Acrislocaliz}
Let $A^+=\Z_p\langle T_1,\dots,T_n\rangle$ and $A=A^+[1/p]$. Further write $X=\Spa (A,A^+)$.\\
\noindent {\rm (i)} The canonical maps
\begin{align*}
A^+\widehat{\otimes}A_{\rm cris}&\longrightarrow \prod k(x)^+\otimes_{\Z_p}A_{\rm cris},\\ 
A\widehat{\otimes}B_{\rm cris}&\longrightarrow \prod k(x)\otimes_{\Q_p}B_{\rm cris},
\end{align*}
where the products runs over all rigid points of $X$, are injections.\\
\noindent {\rm (ii)} A finitely presented $A^+\widehat{\otimes}A_{\rm cris}$-module is zero if it vanishes at all rigid points of $X$.\\
\noindent {\rm (iii)}  A finitely presented $A\widehat{\otimes}B_{\rm cris}$-module is zero if it vanishes at all rigid points of $X$.  \\
\noindent {\rm (iv)} Let $f\in A^+\widehat{\otimes} A_{\rm cris}$ such that $f(x)\in k(x)^+\otimes_{\Z_p}W\subset k(x)^+\otimes_{\Z_p}A_{\rm cris}$ for all rigid points $x\in X$. Then $f\in A^+\otimes_{\Z_p}W\subset A^+\widehat{\otimes}A_{\rm cris}$.
\end{lem}
\begin{proof}
\noindent (i) It suffices to show the statement for $A^+$. Let $f\in A^+\widehat{\otimes}A_{\rm cris}$ such that $f(x)=0$ for all rigid analytic points $x\in X$. We show that 
\[f\mod p^m\in (A^+\widehat{\otimes}A_{\rm cris})/p^m \]
 vanishes for all $m\geq 1$.
But $(A^+\widehat{\otimes}A_{\rm cris})/p^m$ is the quotient of the divided power envelope of $A^+\widehat{\otimes}W(R)$ with respect to $\ker \,\theta_X$ by $p^m$ and $f$ lies in a finitely presented $(A^+\widehat{\otimes}W(R))/p^m$-subalgebra. That means that $f$ lies in the quotient of 
\[\big(\Z_p\langle T_1,\dots,T_s\rangle\widehat{\otimes}W(R)\big)/p^m\]
by some finitely generated ideal and vanishes at all the rigid points of $\Spa(\Q_p\langle T_1,\dots,T_s\rangle,\Z_p\langle T_1,\dots,T_s\rangle)$ for some $s\geq n$. By Remark $\ref{torsionlocaliz}$ is follows that $f=0\mod p^m$.\\
\noindent (ii) Let 
\[\alpha (A^+\widehat{\otimes}A_{\rm cris})^{m_1}\longrightarrow (A^+\widehat{\otimes}A_{\rm cris})^{m_2}\]
be a morphism whose cokernel vanishes at all rigid points. We show that its reduction modulo powers of  $p$ vanish. 
The entries of the matrix of $\alpha\mod p^m$ lie in a finitely presented $(A^+\widehat{\otimes} W(R))/p^m$-algebra and hence a similar argument as above shows that $\alpha$ is surjective modulo powers of $p$.\\
\noindent (iii)  This follows from (ii) after multiplying a presentation with an arbitrary power of $t$ and neglecting the $t$-torsion part of the quotient of two $A^+\widehat{\otimes}A_{\rm cris}$-lattices.\\
\noindent (iv)  By Lemma $\ref{enoughrigpts}$ it suffices to show that $f\in A^+\widehat{\otimes} W(R)$. Reducing the claim modulo powers of $p$ and applying the same argument as above this follows from Remark $\ref{torsionlocaliz}$.
\end{proof}
\begin{rem}
With the notations of the lemma above, let $B$ be a finite $A$-algebra. Then Lemma $\ref{Acrislocaliz}$ (iii) implies that a finitely presented $B\widehat{\otimes}B_{\rm cris}$-module vanishes if it vanishes at all rigid points of $\Spa(A,A^+)$. If $B$ is reduced this is  equivalent to its vanishing at all rigid points of $\Spa(B,B^+)$.\\
Further the analogue of (iv) also holds for $B$, i.e. if $f\in B\widehat{\otimes}B_{\rm cris}$ such that $f(x)\in (k(x)\otimes_AB)\otimes_{\Q_p} K_0\subset (k(x)\otimes_AB)\otimes_{\Q_p}B_{\rm cris}$ for all rigid points $x$, then $f\in B\otimes_{\Q_p}K_0$.
\end{rem}
\begin{cor}\label{Bcrisinvariants}
Let $X\in\Ad^{\rm lft}_{\Q_p}$ and $x\in X$ be a rigid point. Then
\[(\Ocal_X\widehat{\otimes}B_{\rm cris})\otimes k(x)=k(x)\otimes_{\Q_p}B_{\rm cris}.\]
Further we have
\begin{align*}
(\Ocal_X\widehat{\otimes}B_{\rm cris})^{G_K}&=\Ocal_X\otimes_{\Q_p}K_0 \\
 \Fil^0(\Ocal_X\widehat{\otimes}B_{\rm cris})^{\phi=\id}&=\Ocal_X.
\end{align*}
\end{cor}
\begin{proof}
By applying lemma $\ref{Acrislocaliz}$, the proof is the same as in Corollary $\ref{invariantsheaves}$. We reduce to the case where $X$ is affinoid and use the fact that the map
\[\Gamma(X,\Ocal_X\widehat{\otimes}B_{\rm cris})\longrightarrow \prod \Gamma(f^{-1}(y),\Ocal_X)\otimes_{\Q_p}B_{\rm cris}\]
is an injection compatible with $\phi$ and the filtration for a finite morphism $f:X\rightarrow Y$ to a polydisc, where the product is taken over all 
rigid points of $Y$.
\end{proof}

The following diagram summarizes the sheaves of integral period rings on a reduced space $X$. All maps are continuous inclusions.
\begin{equation} \label{periodsheaves}
\begin{aligned}
\begin{xy}
\xymatrix{
\Ocal_X^+\widehat{\otimes}A_{\rm cris} \\
\Ocal_X^+\widehat{\otimes}W(R)\ar@_{(->}[u]\ar@^{(->}[d] & \widetilde{\Acal}_X^{[0,1)}\ar@_{(->}[l]\ar@^{(->}[d] & \Acal_X^{[0,1)}\ar@_{(->}[l]\ar@^{(->}[d]\\
\Ocal_X^+\widehat{\otimes}W(\Frac R) & \widetilde{\Acal}_X\ar@_{(->}[l] & \Acal_X. \ar@_{(->}[l]
}
\end{xy}
\end{aligned}
\end{equation}
Let $\Ecal$ be a vector bundle on an adic space $X$. As $\Ocal_X$ is a sheaf of topological rings, the sections $\Gamma(U,\Ecal)$ of $\Ecal$ over an open subset $U$ have a natural topology. If $G$ is a topological group acting on $\Ecal$ it thus makes sense to ask whether this action is continuous.

\begin{defn}
Let $X$ be an adic space. A \emph{family of crystalline representations over} $X$ is a vector bundle $\Ecal$ on $X$ endowed with a continuous $G_K$-action such that
\[D_{\rm cris}(\Ecal):=(\Ecal\widehat{\otimes}B_{\rm cris})^{G_K} :=\big(\Ecal\otimes_{\Ocal_X}(\Ocal_X\widehat{\otimes}_{\Q_p}B_{\rm cris})\big)^{G_K}\]
is locally on $X$ free of rank $d={\rm rk}\,\Ecal$.
\end{defn}

Let $\Ecal$ be a vector bundle on $X$ with a continuous $G_K$-action. If $x\in X$ is a rigid point then we have an obvious embedding
\[D_{\rm cris}(\Ecal)\otimes k(x)\hookrightarrow D_{\rm cris}(\Ecal\otimes k(x)),\]
compatible with the action of $\phi$ and the filtration, with equality if $\Ecal$ is crystalline.
By a result of Berger and Colmez (see \cite[Corollary 6.33]{BergerColmez}) a family $\Ecal$ of $G_K$-representations on a reduced adic space locally of finite type $X$ is crystalline if and only if the representations on $\Ecal\otimes k(x)$ are crystalline for all rigid analytic points $x\in X$.
\begin{lem}\noindent {\rm (i)} Let $f:X\rightarrow Y$ be a morphism in $\Ad_{\Q_p}^{\lft}$ and $\Ecal$ a vector bundle on $Y$ with continuous $G_K$-action. If $\Ecal$ is crystalline, then so is $f^\ast \Ecal$.\\
\noindent {\rm (ii)} Let $(f_i:X_i\rightarrow X)_i$ be an fpqc cover in $\Ad_{\Q_p}^{\lft}$ and $\Ecal$ be a vector bundle on $\Ecal$ with continuous $G_K$-action. Then $\Ecal$ is crystalline if and only if all the $f_i^\ast \Ecal$ are crystalline.
\end{lem}
\begin{proof}
\noindent (i) We want to check that $f^\ast D_{\rm cris}(\Ecal)=D_{\rm cris}(f^\ast\Ecal)$. The claim is local on $X$ and we may choose a finite morphism $g:X\rightarrow Z$ to a polydisc. Then we have
\begin{align*} \Gamma(X,f^\ast D_{\rm cris}(\Ecal))&\subset \Gamma(X,D_{\rm cris}(f^\ast\Ecal))\\ &\subset \Gamma(X,f^\ast\Ecal\widehat{\otimes}B_{\rm cris})\subset \prod (f^\ast\Ecal\otimes k(z))\otimes_{\Q_p}B_{\rm cris},
\end{align*}
where the product is taken over all closed points of $Z$. The claim follows as $f^\ast D_{\rm cris} (\Ecal)$ is identified with 
\[f^\ast\Ecal\widehat{\otimes}B_{\rm cris}\cap \prod \big((f^\ast\Ecal\otimes k(z))\otimes_{\Q_p}B_{\rm cris}\big)^{G_K},\]
by a similar argument as in the proof of Corollary $\ref{invariantsheaves}$, as $f^\ast D_{\rm cris}(\Ecal)=D_{\rm cris}(f^\ast\Ecal)$ is true in the fibers over the rigid points $z\in Z$, by comparing dimensions.\\
\noindent (ii) By definition we have a $G_K$-equivariant descent datum on the vector bundles $f_i^\ast\Ecal$ and hence also on $f_i^\ast \Ecal\widehat{\otimes}B_{\rm cris}$. By $G_K$-equivariance this induces a descent datum on the $D_{\rm cris}(f_i^\ast\Ecal)$ and by fpqc descent we find a $\Ocal_X\otimes_{\Q_p}K_0$ submodule 
\[D\subset \Ecal\widehat{\otimes}B_{\rm cris}\]
which is locally on $X$ free of rank $d={\rm rk}\,\Ecal$ over $\Ocal_X\otimes_{\Q_p}K_0$ and on which $G_K$ acts trivial. The verification that $D=D_{\rm cris}(\Ecal)$ is similar as in the first part of the lemma.
\end{proof}

\begin{lem}\label{Dcrisvolltreu}
\noindent {\rm (i)} Let $\Ecal$ be a family of crystalline representations on an adic space $X$. Then the canonical morphism
\[\Ecal\longrightarrow V_{\rm cris}(D_{\rm cris}(\Ecal))=\Fil^0(D_{\rm cris}(\Ecal)\widehat{\otimes}_{\Q_p}B_{\rm cris})^{\Phi=\id}\]
is an isomorphism. \\
\noindent {\rm (ii)} Let $\Ecal_1$, $\Ecal_2$ be families of crystalline representations on an adic space $X$. If $D_{\rm cris}(\Ecal_1)\cong D_{\rm cris}(\Ecal_2)$, then $\Ecal_1\cong \Ecal_2$.
\end{lem}
\begin{proof}
As in the classical case (i.e. $X=\Spa(\Q_p,\Z_p)$) this is an easy consequence of $\Fil^0(\Ocal_X\widehat{\otimes}B_{\rm cris})^{\phi=\id}=\Ocal_X$, see \cite[5.2, Theorem (iv)]{Fontaine} for example.
\end{proof}

\subsection{\'Etale $\phi$-modules and $G_{K_\infty}$-representations}

Before we construct a universal crystalline representation we first discuss families of \'etale $\phi$-modules and their relation with families of $G_{K_\infty}$-representations.

\begin{defn}
A \emph{family of} $G_{K_\infty}$\emph{-representations} on an adic space $X\in\Ad_{\Q_p}^{\rm lft}$ is a vector bundle $\Ecal$ on $X$ endowed with a continuous action of $G_{K_\infty}$.
\end{defn}

\begin{defn} Let $X$ be an adic space over $\Q_p$.\\
\noindent (i) Let $\mathscr{R}$ be one of the sheaves $\Acal_X$, $\Bcal_X$, $\widetilde{\Acal}_X$ or $\widetilde{\Bcal}_X$. A \emph{$\phi$-module} over $\mathscr{R}$ is an $\mathscr{R}$-module $N$ that is locally on $X$ free of finite rank over $\mathscr{R}$ together with an isomorphism $\Phi:\phi^\ast N\rightarrow N$.\\
\noindent (ii) Assume that $X$ is reduced. A $\phi$-module over $\Bcal_X$ (resp. $\widetilde{\Bcal}_X$) is called \emph{\'etale}, if it is induced from a $\phi$-module over $\Acal_X$ (resp. $\widetilde{\Acal}_X$).
\end{defn} 

It is not true that every family of \'etale $\phi$-modules comes from a family of $G_{K_\infty}$-representations. However, if $X$ is an adic space and $\mathcal{N}$ is a family of \'etale $\phi$-modules on $X$, we want to show that there is an open subset on which $\mathcal{N}$ is induced from a family of $G_{K_\infty}$-representations.
The main tool is an easy approximation argument which is already contained in \cite[Section 5]{KedlayaLiu} in a slightly different context.

\begin{lem}\label{phieq}
Let $X=\Spa(A,A^+)$ be an reduced, affinoid adic space and fix an element $x\in\Gamma(X,\widetilde{\Acal}_X)=A^+\widehat{\otimes}_{\Z_p}\widetilde{\bf A}$.
Then the equation 
\[y-\phi^{-1}(y)=x\]
has a solution in $\Gamma(X,\widetilde{\Acal}_X)=A^+\widehat{\otimes}_{\Z_p}\widetilde{\bf A}$. Moreover, if $x$ is divisible by $p$, we can choose this solution such that $y$ is divisible by $p$.
\end{lem}
\begin{proof}
First, if $x\in k((u))^{\rm sep}$, then the equation $y-\phi^{-1}(y)=x$ has a solution in $k((u))^{\rm sep}$, as the field is separably closed. If ${\rm val}_R(x)\geq 0$, then we can choose $y$ such that ${\rm val}_R(y)\geq p\,{\rm val}_R(x)$. Hence, by approximation, there is a solution of the equation for all $x$ in the closure of $k((u))^{\rm sep}\subset \Frac R$. The claim follows for $A^+=\Z_p$ by $p$-adic approximation from the modulo $p$-case.

The general statement follows if we expand $x=\sum_i f_i\otimes x_i$ and solve the equations $y_i-\phi^{-1}(y_i)=x_i$. The above claim on the valuations of the solutions modulo $p$ guarantees that the series will converge in $A^+\widehat{\otimes}_{\Z_p}\widetilde{\bf A}$, compare also \cite[Lemma 5.1]{KedlayaLiu}.
\end{proof}

\begin{prop}
Let $X$ be a reduced adic space locally of finite type over $\Q_p$ and $x\in X$ and let $(\mathcal{N},\Phi)$ be an \'etale $\phi$-module of rank $d$ on $X$. \\
\noindent {\rm (i)} If 
\[\dim_{\widehat{k(x)}}(\iota_x^\ast\Ncal)^{\Phi=\id}=d,\]
then $\dim_{k(x)}(\mathcal{N}\otimes k(x))^{\Phi=\id}=d$. \\
\noindent {\rm (ii)} The set 
\[U=\{x\in X\mid \dim_{\widehat{k(x)}}(\iota_x^\ast\mathcal{N})^{\Phi=\id}=d\}\]
is open in $X$ and $\mathcal{N}|_U^{\Phi=\id}$ is a family of $G_{K_\infty}$-representations on $U$.
\end{prop}
\begin{proof}
\noindent (i) Let $\bar f_1,\dots,\bar f_d\in \iota_x^\ast\mathcal{N}$ be a basis on which $\Phi$ acts as the identity.
As $\mathcal{N}\otimes k(x)$ is dense in $\iota_x^\ast\mathcal{N}$ we can approximate this basis by elements $\bar g_1,\dots,\bar g_d\in\mathcal{N}\otimes k(x)$.
Let $\bar B\in \GL_d(\widetilde{\Bcal}_X\otimes k(x))$ denote the matrix of $\Phi$ in this basis. As the $g_i$  are close to the basis $f_i$, the entries of $B$ are close to the entries of the matrix of $\Phi$ in the basis $(f_i)$ and hence the entries of $B-\id$ are arbitrary small
and we may assume that they are divisible by $p$. 
Let $V=\Spa(A,A^+)$ be an affinoid neighbourhood of $x$ such that $\bar B$ lifts to a matrix $B\in \Gamma(V,\GL_d(\widetilde{\Bcal}_X))$ such that $B-\id$ is divisible by $p$. 
Let $(\mathcal{N}',\Phi')$ denote the free $\phi$-module over $\widetilde{\Bcal}_V$ with basis $g_1,\dots,g_d$ on which $\Phi'$ acts by $B$.
By Lemma $\ref{phieq}$ we may choose a matrix $X\in {\rm Mat}_{d\times d}(p\,\Gamma(V,\widetilde{\Acal}_V))$ such that $X-\phi^{-1}(X)=B-\id$.
Changing the basis $g_1,\dots, g_d$ by $\id+X$ the matrix of $\Phi'-\id$ in the new basis is divisible by $p^2$.
As $A$ is $p$-adically complete this process converges to a $\Phi'$-stable basis of $\mathcal{N'}$.
The claim now follows from the equality of $\phi$-modules
\[(\mathcal{N}\otimes k(x),\Phi)=(\mathcal{N}'\otimes k(x),\Phi').\]

\noindent (ii)
By the first step of the proposition we find
\[U=\{x\in X\mid \dim_{k(x)}(\mathcal{N}\otimes k(x))^{\Phi=\id}=d\}.\]
Let $x\in X$ and $ f_1,\dots, f_d$ be lifts of generators of 
\[(\mathcal{N}\otimes k(x))^{\Phi=\id}\]
to some affinoid neighbourhood $V$. After shrinking $V$ we may assume that $f_1,\dots,f_d$ are linearly independent and generate an \'etale lattice in $\mathcal{N}|_V$.
Denote by $B$ the matrix of $\Phi$ in this basis. Then $B-\id$ vanishes at $x$ and has bounded entries. Hence by Lemma $\ref{bouningLemma}$ we may shrink $V$ such that the entries of $B-\id$ are divisible by $p$. By Lemma $\ref{phieq}$ we can choose a matrix $X$ which is divisible by $p$ such that $X-\phi^{-1}(X)=B-\id$.
Then, changing the basis by $\id+X$ we find a basis such that the matrix of $\Phi-\id$ is divisible by $p^2$.
This process converges to a $\Phi$-stable basis of $\mathcal{N}|_V$. This yields the claim.
\end{proof}

\begin{cor}\label{GKinftyvb}
Let $\mathcal{N}$ be an \'etale $\phi$-module of rank $d$ over $\Bcal_X$. The set 
\[U=\{x\in X\mid \dim_{\widehat{k(x)}}\Hom_{\Bcal_{\widehat{k(x)}},\Phi}(\iota_x^\ast\mathcal{N},\widetilde{\Bcal}_{\widehat{k(x)}})=d\}\]
is open and the sheaf $\sheafHom_{\Bcal_U,\Phi}(\mathcal{N}|_U,\widetilde{\Bcal}_U)$ is a vector bundle on $U$.
Further, for a $\Phi$-stable $\Acal_X$-lattice $\mathfrak{N}\subset \mathcal{N}$, the $\Ocal_U^+$-module $\sheafHom_{\Acal_U,\phi}(\mathfrak{N}|_U,\widetilde{\Acal}_U)$ is locally on $U$ free of rank $d$. 
\end{cor}
\begin{proof}
This is the dual statement of the Proposition above.
\end{proof}

The next Proposition shows that the open locus constructed in the Corollary above is compatible with base change and hence this construction is \emph{geometric}.
\begin{prop}\label{maxrankgeometric}
Let $f:X\rightarrow Y$ be a morphism of adic spaces in $\Ad_{\Q_p}^{\rm lft}$ with $Y$ reduced and let $\mathcal{N}_Y$ be an \'etale $\Bcal_Y$-module of rank $d$.
Write $\mathcal{N}_X=f^\ast\mathcal{N}_Y$ for the pullback to $X$ and let 
\begin{align*}
V&=\{y\in Y\mid \dim_{\widehat{k(y)}}\Hom_{\Bcal_{\widehat{k(y)}},\Phi}(\iota_y^\ast\mathcal{N}_Y,\widetilde{\Bcal}_{\widehat{k(y)}})=d\}\\
U&=\{x\in X\mid \dim_{\widehat{k(x)}}\Hom_{\Bcal_{\widehat{k(x)}},\Phi}(\iota_x^\ast\mathcal{N}_X,\widetilde{\Bcal}_{\widehat{k(x)}})=d\}.
\end{align*}
Then $U=f^{-1}(V)$ and $\sheafHom_{\Bcal_U,\Phi}(\mathcal{N}_X|_U,\widetilde{\Bcal}_U)=f^\ast \sheafHom_{\Bcal_V,\Phi}(\mathcal{N}_Y|_V,\widetilde{\Bcal}_V)$.
\end{prop}
\begin{proof}
The inclusion $f^{-1}(V)\subset U$ is obvious. Let $x\in U$ mapping to $y\in Y$. We have to show $y\in V$, i.e.
\[\dim_{\widehat{k(y)}}\Hom_{\Bcal_{\widehat{k(y)}},\Phi}(\iota_y^\ast\mathcal{N}_Y,\widetilde{\Bcal}_{\widehat{k(y)}})=d.\]
The set $U$ is open and $f$ is (locally) of finite type, hence there exists a point $x'\in U\cap f^{-1}(y)$ such that $\widehat{k(x')}$ is a finite extension of $\widehat{k(y)}$. In this case 
\[\iota_{x'}^\ast\mathcal{N}_X=(\iota_y^\ast\mathcal{N}_Y)\otimes_{\widehat{k(y)}}\widehat{k(x')},\]
where we do not need to complete the tensor product as the extension is finite.
Hence we have a descent datum on $\iota_{x'}^\ast\mathcal{N}_X$ which is compatible with the action of $\Phi$. It follows that
\[\Hom_{\Bcal_{\widehat{k(x')}},\Phi}(\iota_{x'}^\ast\mathcal{N},\widetilde{\Bcal}_{\widehat{k(x')}})\]
descends to $\widehat{k(y)}$ which yields the claim on the dimension.
To prove the last claim we again chose locally on $X$ a finite morphism $g:X\rightarrow Z$ to a polydisc. The claim follows from the fact that (assuming locally that $V$ is affinoid and $U'\subset U$ is affinoid) the inclusion
\[\Hom_{\Bcal_V,\Phi}(\Ncal_Y|_V,\widetilde{\Bcal}_V)\widehat{\otimes}_{\Gamma(V,\Ocal_Y)}\Gamma(U',\Ocal_X)\subset \Hom_{\Bcal_{U'},\Phi}(\Ncal_X|_{U'},\widetilde{\Bcal}_{U'})\]
induces an equality in the fibers $g^{-1}(z)$ of $g$ for all rigid points $z\in Z$, which can be seen by comparing dimensions over $\Q_p$.
\end{proof}

\subsection{Construction of a crystalline family}

Now we want to use the results of the preceding subsection to construct an open substack $\Dfrak_\nu^{\rm adm}\subset \Dfrak_\nu^{\rm int}$ on which there exists a family of crystalline representations giving rise to the restriction of the universal filtered $\phi$-module.
The main point is to show that a family of $G_{K_\infty}$-representations and an extension of the associated \'etale $\phi$-module to a certain $\phi$-module on the open unit disc "glue" together to a family of $G_K$-representations.

\begin{prop}\label{crystallinefamily}
Let $X\in\Ad_{\Q_p}^{\lft}$ be reduced and $(\Mfrak,\Phi)$ be a $\phi$-module of height $1$ and rank $d$ over $\Acal_X^{[0,1)}$ with associated filtered $\phi$-module $(D,\Phi,\Fcal^\bullet)$. Then 
\begin{align*}
U&:=\{x\in X\mid {\rm rk}_{k(x)^+}\Hom_{\Acal_X^{[0,1)}\otimes k(x),\Phi}(\Mfrak\otimes k(x), \widetilde{\Acal}_X^{[0,1)}\otimes k(x))=d\}\\
&\,= \{x\in X\mid \dim_{\widehat{k(x)}}\Hom_{\Bcal_{\widehat{k(x)}},\Phi}(\iota_x^\ast (\Mfrak\otimes_{\Acal_X^{[0,1)}}\Bcal_X),\widetilde{\Bcal}_{\widehat{k(x)}})=d\}
\end{align*}
 and this set is open and 
\[\sheafHom_{\Acal_U^{[0,1)},\Phi}(\Mfrak|_U,\widetilde{\Acal}_U^{[0,1)})\otimes_{\Z_p}\Q_p=\sheafHom_{\Bcal_U,\Phi}(\Mfrak|_U\otimes_{\Acal_U^{[0,1)}}\Bcal_U,\widetilde{\Bcal}_U)\]
is a vector bundle on $U$.
\end{prop}
\begin{proof}
We have a canonical injection
\begin{equation}\label{overconvergence}
\sheafHom_{\Acal_X^{[0,1)},\Phi}(\Mfrak,\widetilde{\Acal}_X^{[0,1)})\hookrightarrow \sheafHom_{\Acal_X,\Phi}(\Mfrak\otimes_{\Acal_X^{[0,1)}}\Acal_X,\widetilde{\Acal}_X).
\end{equation}
By Corollary $\ref{GKinftyvb}$ the set 
\begin{align*}
U'&=\{x\in X\mid \dim_{\widehat{k(x)}}\Hom_{\Bcal_{\widehat{k(x)}},\Phi}(\iota_x^\ast (\Mfrak\otimes_{\Acal_X^{[0,1)}}\Bcal_X),\widetilde{\Bcal}_{\widehat{k(x)}})=d\}
\end{align*}
is open and the sheaf 
\[\sheafHom_{\Acal_X,\Phi}(\Mfrak\otimes_{\Acal_X^{[0,1)}}\Acal_X,\widetilde{\Acal}_X)\] 
restricts to a locally free $\Ocal_X^+$-module on $U'$.\\
We claim that $U=U'$ and that $(\ref{overconvergence})$ is an isomorphism on $U$.
The inclusion $U\subset U'$ is obvious and we have to prove the converse. Let 
\begin{align*}
f\in&\ \Gamma(V,\sheafHom_{\Acal_X,\Phi}(\Mfrak\otimes_{\Acal_X^{[0,1)}}\Acal_X,\widetilde{\Acal}_X))\\
&= \Hom_{\Gamma(V,\Acal_X),\Phi}(\Gamma(V,\Mfrak\otimes_{\Acal_X^{[0,1)}}\Acal_X),\Gamma(V,\widetilde{\Acal}_X))
\end{align*}
for some affinoid $V\subset U'$. This equality is true, as $\Mfrak$ is locally free. For some $m\in \Gamma(V,\Mfrak)\subset \Gamma(V,\Mfrak\otimes_{\Acal_X^{[0,1)}}\Acal_X)$ we then have 
\[f(m)(x)\in\widetilde{\Acal}_X^{[0,1)}\otimes k(x)\]
for all rigid analytic points $x\in V$ by \cite[Corollary 2.1.4]{crysrep}, and hence we find that $f(m)\in \Gamma(V,\widetilde{\Acal}_X^{[0,1)})$ by Lemma $\ref{enoughrigpts}$, as $X$ is reduced. It follows that $f$ restricts to an element of 
\[\Hom_{\Gamma(X,\Acal_X^{[0,1)}),\Phi}(\Gamma(V,\Mfrak),\Gamma(V,\widetilde{\Acal}_X^{[0,1)})).\]
\end{proof}

\begin{prop}\label{GKinftytoGK}
Let $X$ be a reduced space and $(\Mfrak,\Phi)$ be a $\phi$-module of height $1$ and rank $d$ over $\Acal_X^{[0,1)}$ such that 
\[{\rm rk}_{k(x)^+}\Hom_{\Acal_X^{[0,1)}\otimes k(x),\Phi}(\Mfrak\otimes k(x), \widetilde{\Acal}_X^{[0,1)}\otimes k(x))=d\]
for all $x\in X$. Write $(D,\Phi,\Fcal^\bullet)$ for the filtered $\phi$-module associated to $(\Mfrak,\Phi)$ by the period morphism. There is a canonical isomorphism of $\Ocal_X$-modules
\begin{equation}\label{VcrisVB}
\sheafHom_{\Acal_X^{[0,1)},\Phi}(\Mfrak,\widetilde{\Acal}_X^{[0,1)})\otimes_{\Z_p}\Q_p\longrightarrow \sheafHom_{\Phi,\Fil}(D,\Ocal_X\widehat{\otimes}B_{\rm cris}).\end{equation}
Further $\mathcal{E}=V_{\rm cris}^\ast(D):=\sheafHom_{\Phi,\Fil}(D,\Ocal_X\widehat{\otimes}B_{\rm cris})$ is a vector bundle of rank $d$ on $X$ with a continuous action of $G_K$ such that there is a canonical isomorphism of filtered $\phi$-modules
\[\sheafHom_{\Ocal_X[G_K]}(\mathcal{E},\Ocal_X\widehat{\otimes}B_{\rm cris})\cong (D,\Phi,\Fcal^\bullet).\]
\end{prop}
\begin{proof}
The point of this proposition is that \cite[Proposition 2.1.5]{crysrep} (see also \cite[Proposition 11.3.3]{survey}) works also in the relative case. We give the definition of the map and refer to \cite{survey} for the details.\\
Write 
\[\Mcal=\Mfrak\otimes_{\Acal_X^{[0,1)}}\Bcal_X^{[0,1)}\]
for the vector bundle on $X\times\Ubb$ associated with $\Mfrak$.\\
First the map $\widetilde{\bf A}^{[0,1)}\hookrightarrow W(R)\hookrightarrow B_{\rm cris}$ (compare also $(\ref{periodsheaves})$) induces injections
\begin{align*}
\sheafHom_{\Acal_X^{[0,1)},\Phi}(\Mfrak,\widetilde{\Acal}_X^{[0,1)})\otimes_{\Z_p}\Q_p& \hookrightarrow \sheafHom_{\Acal_X^{[0,1)},\Phi}(\Mfrak,\Ocal_X\widehat{\otimes}B_{\rm cris})\\ &\hookrightarrow \sheafHom_{\Bcal_X^{[0,1)},\Phi}(\Mcal,\Ocal_X\widehat{\otimes}B_{\rm cris}).
\end{align*}
Secondly consider the map
\begin{align*}
\sheafHom_{\Bcal_X^{[0,1)},\Phi}(\Mcal,\Ocal_X\widehat{\otimes}B_{\rm cris}) & \rightarrow \sheafHom_{\Bcal_X^{[0,1)},\Phi}(D\otimes_{K_0}\Bcal_X^{[0,1)},\Ocal_X\widehat{\otimes}B_{\rm cris})\\ &=
\sheafHom_\Phi(D,\Ocal_X\widehat{\otimes}B_{\rm cris}),
\end{align*}
given by composing with the $\Phi$-compatible injection
\[\xi: D\otimes_{K_0}\Bcal_X^{[0,1)}\hookrightarrow \Mcal\]
of Lemma $\ref{lemxi}$. The resulting map $(\ref{VcrisVB})$ is injective, as the source is a vector bundle on a reduced space and the kernel vanishes at all rigid points. Further the map has image in 
\[\sheafHom_{\Phi,\Fil}(D,\Ocal_X\widehat{\otimes}B_{\rm cris})\subset \sheafHom_\Phi(D,\Ocal_X\widehat{\otimes}B_{\rm cris})\]
which can be seen as follows: The inclusion is true for all rigid points by the argument in the proof of \cite[Proposition 11.3.3]{survey} and, as $X$ is reduced, locally on $X$ the sections of  $\sheafHom_{\Phi,\Fil}(D,\Ocal_X\widehat{\otimes}_{\Q_p}B_{\rm cris})$ are identified with
\[\Hom_{\Phi}(D,\Ocal_X\widehat{\otimes}B_{\rm cris})\cap \prod \Hom_{\Phi,\Fil}(D\otimes k(x),k(x)\otimes_{\Q_p}B_{\rm cris}),\]
where the product runs over all rigid points $x\in X$.

Now $V_{\rm cris}^\ast(D)$ is a vector bundle on $X$ which has a continuous $G_K$-action induced from the action on $B_{\rm cris}$.

Finally the evaluation map
\[D\rightarrow \sheafHom_{\Ocal_X[G_K]}(\sheafHom_{\Phi,\Fil}(D,\Ocal_X\widehat{\otimes}B_{\rm cris}),\Ocal_X\widehat{\otimes}B_{\rm cris})\]
defines an isomorphism: The map is clearly injective, as $D$ is a coherent sheaf on a reduced space and the kernel vanishes at all rigid points. Further the cokernel of the induced map of finite $\Ocal_X\widehat{\otimes}B_{\rm cris}$-modules
\[D{\otimes}_{K_0}B_{\rm cris}\rightarrow \sheafHom_{\Ocal_X}(\sheafHom_{\phi,\Fil}(D,\Ocal_X\widehat{\otimes}B_{\rm cris}),\Ocal_X\widehat{\otimes}B_{\rm cris})\]
vanishes at all rigid points and by the usual argument we conclude that it is an isomorphism. The claim follows from this after taking $G_K$-invariants on both sides.
\end{proof}

\begin{theo}
The groupoid that assigns to an adic space $X\in\Ad_{\Q_p}^{\lft}$ the groupoid of triples $(D,\Phi,\Fcal^\bullet)\in\Dfrak_\nu^{\rm int}(X)$ such that
\[\dim_{\widehat{k(x)}}(\iota_x^\ast \Mfrak\otimes_{\Acal^{[0,1)}_{\widehat{k(x)}}}\widetilde{\Bcal}_{\widehat{k(x)}})^{\Phi=\id}=d \]
for all $x\in X$ and any choice of an integral model $\Mfrak\subset \underline{\Mcal}(D)$ in a some neighbourhood $U$ of $x$ {\rm (}where $\underline{\Mcal}(-)$ is the map $(\ref{mapM})${\rm )}, 
is an open substack $\Dfrak_\nu^{\rm adm}\subset \Dfrak_\nu^{\rm int}$. For any finite extension $F$ of $\Q_p$ we have $\Dfrak_\nu^{\rm adm}(F)=\Dfrak_\nu^{\rm int}(F)=\Dfrak_\nu^{\rm wa}(F)$. Further there exists a family of crystalline representations $\Ecal_\nu^{\rm univ}$ on $\Dfrak_\nu^{\rm adm}$ such that
\[D_{\rm cris}(\Ecal_\nu^{\rm univ})=(D,\Phi,\Fcal^\bullet)|_{\Dfrak_\nu^{\rm adm}}.\]
\end{theo}
\begin{proof}
It is clear that the condition
\[\dim_{\widehat{k(x)}}(\iota_x^\ast \Mfrak\otimes_{\Acal^{[0,1)}_{\widehat{k(x)}}}\widetilde{\Bcal}_{\widehat{k(x)}})^{\Phi=\id}=d\]
is independent of the choice of an integral model in a neighbourhood of $x$. Further this condition is fpqc-local by Proposition $\ref{maxrankgeometric}$.
By the Proposition above we can define the sheaf $\Ecal_\nu^{\rm univ}$ in the universal case, i.e the open subspace of $\Res_{K_0/\Q_p}\GL_d\times\Gr_{K,\nu}$ defined by the condition in the theorem, as this is reduced. This vector bundle is well defined as $D_{\rm cris}$ commutes with pullbacks and as two crystalline representations $\Ecal_1$ and $\Ecal_2$ are isomorphic if $D_{\rm cris}(\Ecal_1)$ and $D_{\rm cris}(\Ecal_2)$ are isomorphic (see Lemma \ref{Dcrisvolltreu}). Now $\Ecal_\nu^{\rm univ}$ can be defined in general by pullback from the universal case.
\end{proof}

\subsection{Universality of the admissible locus}
We now want to show that the substack constructed above with its crystalline family is indeed the \emph{stack of crystalline representations}.
The main point is to show that the $\phi$-module associated to the restriction of a crystalline family to $G_{K_\infty}$ is overconvergent.

Let $\nu$ be a dominant cocharacter of the algebraic group $\Res_{K/\Q_p}\GL_d$ as in $(\ref{specialcochar})$, defined over $E$.
We say that a family of crystalline representations $\Ecal$ on $X$ has Hodge-Tate weights $\nu$ if the filtration on $D_{\rm cris}(\Ecal)\otimes_{K_0}K$ is of type $\nu$. 
\begin{theo}
The groupoid 
\[\underline{\rm Rep}_{\rm cris}^{\nu}:X \mapsto\left\{
{\begin{array}{*{20}c}
\text{ families of crystalline representations on}\ X\\ \text{with Hodge-Tate weights}\ \nu
\end{array}}\right\}\]
on the category $\Ad_{E}^{\rm lft}$ is isomorphic to the stack $\Dfrak_\nu^{\rm adm}$ and hence is an open substack of $\Dfrak_\nu^{\rm wa}$. 
Especially it is an Artin stack over $\Q_p$. Further it is contained in the image of the period morphism
and contains all rigid analytic points of $\Dfrak_\nu^{\rm wa}$.
\end{theo}
\begin{proof}
The injectivity of the morphism $D_{\rm cris}$ from the stack of crystalline representations with Hodge-Tate weights $\nu$ to $\Dfrak_\nu$ follows from Lemma \ref{Dcrisvolltreu}. The rest of the claim is an immediate consequence of the propositions below.
\end{proof}

\begin{prop}\label{overconvforfam}
The morphism 
\[D_{\rm cris}:\underline{\rm Rep}_{\rm cris}^{\nu}\longrightarrow \Dfrak_\nu\]
factors over the image of the period morphism $\Dfrak_\nu^{\rm int}\subset \Dfrak_\nu$.
\end{prop}
\begin{proof}
Let $X$ be an adic space and $\Ecal$ a family of crystalline representations over $X$. Write $D_{\rm cris}(\Ecal)=(D,\Phi,\Fcal^\bullet)$ for the associated filtered $\phi$-module on $X$. As $\Dfrak_\nu^{\rm int}\subset \Dfrak_\nu$ is open and factorizing over an open subspace may be checked point wise, we may assume that $X$ is reduced.
We consider the inclusions 
\begin{align*}
 \Bcal_X^{[0,1)}(\Ecal\widehat{\otimes}B_{\rm cris})^{G_K}&\subset \Mcal=\underline{\Mcal}(D) \\ 
 & \subset \lambda^{-1}\Bcal_X^{[0,1)}(\Ecal\widehat{\otimes}B_{\rm cris})^{G_K}\subset \Ecal\widehat{\otimes}B_{\rm cris}, 
\end{align*}
where $\lambda$ is the element defined in $(\ref{lambda})$.\\
Then the $G_{K_\infty}$-action on $\Mcal\subset \Ecal\widehat{\otimes}_{\Q_p}B_{\rm cris}$ is trivial as this is true for all rigid analytic points $x\in X$.
Hence 
\begin{equation*}\label{equation1}
\Mcal\otimes_{\Bcal_X^{[0,1)}}\Bcal_X^{[r,s]}\subset (\Ecal{\otimes}_{\Ocal_X}\widetilde{\Bcal}_X^{[r,s]})^{G_{K_\infty}}
\end{equation*}
for some $0<r<r^{1/p^2}\leq s<1$ near the boundary and where 
\[\widetilde{\Bcal}_X^{[r,s]}=\widetilde{\Acal}_X^{[0,1)}\widehat{\otimes}_{\Acal_X^{[0,1)}}\Bcal_X^{[r,s]}.\]
Now 
\[\Mcal\otimes_{\Bcal_X^{[0,1)}}\widetilde{\Bcal}_X^{[r,s]}=\Ecal\otimes_{\Ocal_X}\widetilde{\Bcal}_X^{[r,s]},\]
as the cokernel vanishes at all rigid analytic points and $X$ is reduced.
We find that a $G_{K_\infty}$-invariant $\Ocal_X^+$ lattice $\Ecal^+\subset \Ecal$ (which exists locally on $X$) defines an integral structure in 
\[\Mcal\otimes_{\Bcal_X^{[0,1)}}\Bcal_X^{[r,s]}=\big(\Mcal\otimes_{\Bcal_X^{[0,1)}}\widetilde{\Bcal}_X^{[r,s]}\big)^{G_{K_\infty}}=\big(\Ecal\otimes_{\Ocal_X}\widetilde{\Bcal}_X^{[r,s]}\big)^{G_{K_\infty}}.\]
Now the claim follows from Theorem $\ref{maintheoint}$.
\end{proof}

\begin{prop}
The morphism 
\[D_{\rm cris}:\underline{\rm Rep}_{\rm cris}^{\nu}\longrightarrow \Dfrak_\nu\]
factors over the admissible locus $\Dfrak_\nu^{\rm adm}\subset \Dfrak_\nu$.
\end{prop}
\begin{proof}
Let $X$ be an adic space and $\Ecal$ a family of crystalline representations on $X$ with $D_{\rm cris}(\Ecal)=(D,\Phi,\Fcal^\bullet)$. Again we may assume that $X$ is reduced. By Proposition $\ref{overconvforfam}$ we may choose locally on $X$ a $\phi$-module $(\Mfrak,\Phi)$ of height 1 over $\Acal_X^{[0,1)}$ (compare Definition $\ref{height1modules}$) which is an integral model for $(D,\Phi,\Fcal^\bullet)$. Then the canonical map
\[\Ecal\longrightarrow (\Mfrak\otimes_{\Acal_X^{[0,1)}}\widetilde{\Bcal}_X)^{\Phi=\id}\]
is an isomorphism of $G_{K_\infty}$-representations. Indeed, it is injective, as $\Ecal$ is a vector bundle on a reduced space and the kernel vanishes at all rigid analytic points. Further the cokernel of the induced map of finite $\widetilde{\Bcal}_X$-modules
\[\Ecal_X\otimes_{\Ocal_X}\widetilde{\Bcal}_X\longrightarrow \Mfrak\otimes_{\Acal_X^{[0,1)}}\widetilde{\Bcal}_X\]
vanishes at all rigid points of $x$. As $X$ is reduced this map is an isomorphism by the argument of Corollary $\ref{invariantsheaves}$. The claim now follows after taking $\Phi$-invariants in both sides.
\end{proof}

We end this section by giving two examples of a crystalline families and the corresponding filtered $\phi$-modules.
\begin{expl}\label{unramchars}
Let $X=\partial\boldB=\Spa(\Q_p\langle T,T^{-1}\rangle,\Z_p\langle T,T^{-1}\rangle)$ and $K=\Q_p$. Consider the $1$-dimensional filtered $\phi$-module $D=\Ocal_X$
with $\Phi$ given by multiplication with $T$ and $\Fcal^i=D$ if $i\leq 0$ and $\Fcal^i=0$ if $i>0$ \footnote{This example is due to G. Chenevier in the context of overconvergent $(\phi,\Gamma)$-modules, see \cite{KedlayaLiu} for example.}.
Then this family is weakly admissible and has an integral model on $X$. The fiber at a rigid point $a\in\Ocal_{\bar\Q_p}^\times$ gives rise to the unramified character mapping the Frobenius to $a$.
But this does not define a continuous character at the Gauss point $\eta\in X$, as the character $\Z\rightarrow k(\eta)^\times$ mapping $1\in\Z$ to $T$ has no continuous extension to $\widehat{\Z}$. In fact the admissible locus is given by
\[X^{\rm adm}=X\backslash \overline{\{\eta\}}\cong \coprod \Ubb,\]
where the disjoint union of open unit discs is indexed by the closed points of $\Gbb_{m,\,\Fbb_p}$.\\
Especially we see that $\Dfrak_\nu^{\rm adm}\subsetneq \Dfrak_\nu^{\rm int}$ in general.
\end{expl}
\begin{expl}
In this example we want to show that in general the admissible locus is not just a disjoint union of deformation spaces of residual Galois representations (or rather some quotient of this), as in the first example.\\
Let $X=\boldB=\Spa(\Q_p\langle T\rangle,\Z_p\langle T\rangle)$ and $K$ a ramified extension of $\Q_p$ with $E(u)\in\Z_p[u]$ the minimal polynomial of some uniformizer. Consider the integral family given by  $\Mfrak=\Acal_X^{[0,1)} e_1\oplus \Acal_X^{[0,1)} e_2$ and 
\[\Phi\begin{pmatrix}e_1 \\ e_2\end{pmatrix}=\begin{pmatrix} 1 & T \\ 0 & E(u)\end{pmatrix}\begin{pmatrix} e_1 \\ e_2 \end{pmatrix}.\]
For a rigid point $x$ this is an extension of the cyclotomic character with the trivial character which splits over the origin.
But any neighbourhood of the origin contains points, where this extension does not split, and the representations on the punctured unit disc are isomorphic.
Hence the whole family is admissible and the residual representation is non constant. 
\end{expl}

\medskip
\noindent
\address{Mathematisches Institut der Universit\"at Bonn\\ Endenicher Allee 60, 53115 Bonn, Germany}\\
\email{hellmann@math.uni-bonn.de}


\begin{thebibliography}{5mm}


\bibitem[BrCo]{survey}
O. Brinon, B. Conrad,
\emph{$p$-adic Hodge theory},
notes from the CMI summer school 2009.

\bibitem[BeCo]{BergerColmez}
L. Berger, P. Colmez,
\emph{Familles des repr\'esentations de de Rham et monodromie $p$-adique},
in: Repr\'esentations $p$-adiques I: repr\'esentations galoisiennes et $(\phi,\Gamma)$-modules, Asterisque {\bf 319} (2008), 303-337.

\bibitem[Be]{Berko}
V. Berkovich,
\emph{Spectral theory and analytic geometry over non-archimedian fields},
Math. surveys and monographs {\bf 33}, American Math. Soc. 1990.


\bibitem[BG]{BoschGoertz}
S. Bosch, U. G\"ortz,
\emph{Coherent modules and their descent on relative rigid spaces},
J. Reine Angew. Math. {\bf 495} (1998), 119-134.


\bibitem[BGR]{BoschGR}
S. Bosch, U. G\"untzer, R. Remmert,
\emph{Non-archimedian analysis},
Springer-Verlag, Berlin 1984.


\bibitem[Br1]{pdivBreuil}
C. Breuil,
\emph{Groupes $p$-divisibles, groupes finies et modules filtr\'es},
Ann of Math. {\bf 152} (2000), 289-549.

\bibitem[Br2]{Breuil}
C. Breuil,
{\emph {Integral $p$-adic Hodge Theory}},
Algebraic Geometry 2000, Azumino, Adv. Studies in Pure Math. {\bf 36}, 2002, 51-80.

\bibitem[CF]{ColmezFont}
P. Colmez, J.-M. Fontaine,
\emph{Constructions des repr\'esentations $p$-adiques semi-stables},
Invent. Math. {\bf 140} (2000), 1-43.

\bibitem[DOR]{DatOrlikRapo}
J-F. Dat, S. Orlik, M. Rapoport,
\emph{Period domains over finite and $p$-adic fields},
Cambridge Tracts in Math. {\bf 183}, Cambridge University Press 2010.

\bibitem[Fa]{Faltings}
G. Faltings,
\emph{Coverings of $p$-adic period domains},
J. reine angew. Math. {\bf 643} (2010), 111-139.

\bibitem[FF]{FargFont}
L. Fargues, J.-M. Fontaine,
\emph{Courbes et fibr\'es vectoriels en th\'eorie de Hodge $p$-adique},
preprint, 2010.

\bibitem[Fo]{Fontaine}
J.-M. Fontaine,
\emph{Sur certains types des repr\'esentations $p$-adiques du groupe d'Galois d'un corps local; construction d'un anneau de Barsotti-Tate},
Ann. of Math. {\bf 115} (1982), 529-577.

\bibitem[Gr]{Gruson}
L. Gruson,
\emph{Fibr\'es vectoriels sur un polydisque ultram\'etrique},
Ann. Sci. \'Ecole Norm. Sup. (4) {\bf 1} (1968), 45-89.

\bibitem[Ha1]{Hartl1}
U. Hartl,
\emph{On a conjecture of Rapoport and Zink},
Preprint 2006. arXiv:math.NT/0709.3444v3

\bibitem[Ha2]{Hartl2}
U.Hartl,
\emph{On period spaces of $p$-divisible groups},
Comptes Rendus Math\'ematique Acad. Sci. Paris, Ser. I {\bf 346} (2008), 1123-1128.

\bibitem[He]{Hellmann}
E. Hellmann,
\emph{On families of weakly admissible filtered $\phi$-modules and the adjoint quotient of $\GL_d$},
preprint 2011, 	arXiv:1102.0119v1.

\bibitem[Hu1]{contval}
R. Huber,
\emph{Continuous valuations},
Math. Z. {\bf 212} (1993), 445-447.

\bibitem[Hu2]{Hu2}
R. Huber,
\emph{A generalization of formal schemes and rigid analytic varieties},
Math. Z. {\bf 217} (1994), 513-551.

\bibitem[Hu3]{Huber}
R. Huber,
\emph{\'Etale cohomology of rigid analytic varieties and adic spaces},
Aspects of Math., E30, Friedr. Vieweg \& Sohn, Braunschweig, 1996.

\bibitem[Ke]{Kedlaya}
K.Kedlaya,
\emph{Slope filtrations for relative relative Frobenius},
in: Repr\'esentations $p$-adiques I: repr\'esentations galoisiennes et $(\phi,\Gamma)$-modules, Asterisque {\bf 319} (2008), 259-301.

\bibitem[KL]{KedlayaLiu}
K.Kedlaya, R. Liu,
\emph{On families of $(\phi,\Gamma)$-modules},
to appear in Algebra and Number Theory.

\bibitem[Ki1]{Kisin}
M. Kisin,
{\emph {Moduli of finite flat group schemes, and modularity}},
Ann. of Math. {\bf 107} no. 3 (2009), 1085-1180

\bibitem[Ki2]{crysrep}
M. Kisin,
\emph{Crystalline representations and $F$-crystals},
in: Algebraic geometry and number theory, 459-496, Progr. Math., 253, Birkh\"auser Boston, Boston MA, 2006.

\bibitem[Ki3]{potsemdeform}
M. Kisin,
\emph{Potentially semi-stable deformation rings},
J. Amer. Math. Soc. {\bf 21} no. 2 (2008), 513-546. 

\bibitem[Ki4]{Kisinp=2}
M. Kisin,
\emph{Modularity of $2$-adic Barsotti-Tate representations},
Invent. Math. {\bf 178} no. 3 (2009), 587-634

\bibitem[Liu]{Liu}
R. Liu,
\emph{Slope filtrations in families},
to appear in J. Inst. Math. Jussieu.

\bibitem[Lue]{Lutke}
W. L\"utkebohmert,
\emph{Vektorraumb\"undel \"uber nicht archimedischen holomorphen R\"aumen},
Math. Z. {\bf 152} (1977), 127-143.

\bibitem[PR]{phimod}
G. Pappas, M. Rapoport,
\emph{$\Phi$-modules and coefficient spaces},
Moscow Math. J. {\bf 9} no. 3 (2009), 625- 664.

\bibitem[RZ]{RapoZink}
M. Rapoport, T. Zink,
\emph{Period spaces for $p$-divisible groups},
Ann. of Math. Studies, vol {\bf 141}, Princeton University Press, Princeton, NJ, 1996.

\end{thebibliography}
\end{document}